\UseRawInputEncoding
\documentclass[final,hidelinks,onefignum,onetabnum]{siamart251216}

\UseRawInputEncoding

\usepackage{lipsum}
\usepackage{amsfonts}
\usepackage{graphicx}
\usepackage{epstopdf}
\usepackage{algorithmic}
\ifpdf
  \DeclareGraphicsExtensions{.eps,.pdf,.png,.jpg}
\else
  \DeclareGraphicsExtensions{.eps}
\fi

\newsiamremark{remark}{Remark}
\newsiamremark{hypothesis}{Hypothesis}
\crefname{hypothesis}{Hypothesis}{Hypotheses}
\newsiamthm{claim}{Claim}
\newsiamremark{fact}{Fact}
\crefname{fact}{Fact}{Facts}

\headers{Globally Convergent Third-Order Newton Method}{Y. CAI, W. ZHU, C. CARTIS, AND G. ZARDINI}

\title{A Globally Convergent Third-Order Newton Method via Unified Semidefinite Programming Subproblems}

\providecommand{\thanksmark}[1]{\footnotemark[#1]}
\author{Yubo Cai\thanks{Laboratory for Information \& Decision Systems, Massachusetts Institute of Technology, Cambridge, MA (\email{yubocai@mit.edu}; \email{gzardini@mit.edu}).}
\and Wenqi Zhu\thanks{Mathematical Institute, University of Oxford, UK (\email{wenqi.zhu@maths.ox.ac.uk}; \email{cartis@maths.ox.ac.uk}).}
\and Coralia Cartis\thanksmark{2}
\and Gioele Zardini\thanksmark{1}
}

\usepackage{amsopn}

\ifpdf
\hypersetup{
  pdftitle={A Globally Convergent Third-Order Newton Method via Unified Semidefinite Programming Subproblems},
  pdfauthor={Y. CAI, W. ZHU, C. CARTIS, AND G. ZARDINI}
}
\fi

\newtheorem{assumption}{Assumption}

\usepackage[numbers,sort&compress]{natbib}
\usepackage{graphicx}
\usepackage{booktabs}
\usepackage{tikz}
\usetikzlibrary{positioning}

\usepackage{enumitem}
\usepackage{array}
\usepackage{placeins}

\usepackage[capitalize]{cleveref}
\crefname{unframedtheorem}{Theorem}{Theorems}
\crefname{lemma}{Lemma}{Lemmas}
\crefname{proposition}{Proposition}{Propositions}
\crefname{corollary}{Corollary}{Corollaries}
\crefname{remark}{Remark}{Remarks}
\crefname{definition}{Definition}{Definitions}
\crefname{assumption}{Assumption}{Assumptions}
\UseRawInputEncoding
\usepackage{xcolor}
\usepackage[acronym, nopostdot, nonumberlist]{glossaries}
\glsdisablehyper

\newacronym{almton}{ALMTON}{Adaptive Levenberg-Marquardt Third-Order Newton Method}
\newacronym{arp}{AR$p$}{adaptive regularization with models of order $p$}
\newacronym{sdp}{SDP}{semidefinite programming}
\newacronym{psd}{PSD}{positive semidefinite}
\newacronym{lm}{LM}{Levenberg-Marquardt}
\newacronym{sos}{SOS}{sums-of-squares}

\newcommand{\taylor}[2]{\Phi^{#1}_{f, #2}}

\begin{document}

\maketitle

\begin{abstract}
  We propose the \gls{almton} method for unconstrained nonconvex optimization;
  to our knowledge, the framework provides the first globally
  convergent realization of the unregularized third-order Newton method.
  Unlike the standard Adaptive Regularization framework with third-order models (AR$3$), which enforces global behavior through a quartic term, \gls{almton} employs an adaptive Levenberg-Marquardt (quadratic) regularization.
  This choice preserves a cubic model at every iteration, so that every subproblem is a tractable \gls{sdp}. In particular, \gls{almton}-Simple constructs exactly one logical SDP subproblem per nonterminal outer iteration, while \gls{almton}-Heuristic may solve several SDP subproblems within its model phase.

  Algorithmically, \gls{almton} follows a mixed-mode strategy: it attempts an unregularized third-order step whenever the cubic Taylor model admits a strict local minimizer with adequate curvature, and activates (or increases) quadratic regularization only when needed to ensure that the model is well posed and the step is globally reliable.
  For the Heuristic strategy, under the stated assumptions and an exact
  local-minimizer oracle, we prove finite termination at an
  $\epsilon$-approximate first-order stationary point with
  $\mathcal{O}(\epsilon^{-2})$ worst-case evaluation complexity.  Moreover, if
  an accepted iterate enters the stated neighborhood of a positive-definite
  local minimizer, subsequent nonterminal steps recover the unregularized
  third-order Newton recursion and its cubic local rate.
  Under a common post hoc terminal audit over 4,500 deterministic starts on five two-dimensional nonconvex problems, both \gls{almton} variants satisfy the terminal criterion on 99.91\% of the instances, compared with 55.42\% for the unregularized third-order Newton method.  This gain is obtained at substantial semidefinite-programming cost: AR$2$-Simple is faster on this test set, so the numerical evidence supports robust globalization of the unregularized cubic model rather than overall empirical dominance.
  We also provide a controlled audit of implementation scaling and identify
  protocol-dependent computational boundaries of the current SDP realization.
\end{abstract}



\section{Introduction}\label{sec:intro}
We consider unconstrained nonconvex optimization problems of the form
\begin{equation} \label{prob:unconstrained_optimization}
  \min _{x \in \mathbb{R}^n} f(x), \vspace{-0.5em}
\end{equation}
which arise in diverse application domains, including neural network training~\cite{goodfellow2016deep,arora2018convergence,bottou2018optimization,berahas2019quasi,kurbiel2017training}, digital filter design~\cite[Chapter~9.4]{antoniou2007practical}, and volatility estimation~\cite[Chapter~6]{cornuejols2018optimization}.
Throughout, we assume that $f:\mathbb{R}^n\to\mathbb{R}$ is $p$-times continuously differentiable (for some $p\ge 1$) and bounded below by a finite value $f_{\mathrm{low}}$.

Algorithms for~\eqref{prob:unconstrained_optimization} are typically iterative, producing a sequence $\{x_k\}$ via updates based on local models of $f$.
The accuracy and geometry captured by these models depend strongly on the order of derivative information used.
First-order methods use only $\nabla f$ and are inexpensive per iteration, while Newton-type methods incorporate $\nabla^2 f$ and often achieve fast local convergence when the quadratic model is reliable.
Higher-order methods, and in particular third-order methods, incorporate derivatives of order $p\ge 3$ to build richer Taylor models that can remain accurate over a larger neighborhood and better reflect nonconvex curvature~\cite{ahmadi2024higher,cartis2020concise}.
This paper focuses on third-order methods.

A central challenge in nonconvex optimization is to combine \emph{local efficiency} with \emph{global reliability}.
For first- and second-order methods, this is commonly achieved through line-search or trust-region globalization~\cite{dennis1996numerical,conn2000trust,yuan2015recent}, which enforce sufficient decrease in the objective even when the local model is inaccurate far from the current iterate.
Extending these ideas to higher-order models is more delicate: higher-degree Taylor polynomials can be nonconvex and even unbounded below, and their minimization subproblems can become substantially harder to solve robustly.

A principled approach to this difficulty is \emph{adaptive regularization with higher-order models}.
The adaptive regularization framework of order $p$ (AR$p$)~\cite{birgin2017worst,cartis2011adaptive} augments the $p$th-order Taylor model with a $(p+1)$st-degree regularizer,
\begin{equation*}
\frac{\sigma_k}{p+1}\,\| x - x_k \|^{p+1}, \qquad \sigma_k > 0,
\end{equation*}
and updates $\sigma_k$ dynamically to guarantee progress.
Under standard assumptions, AR$p$ methods admit worst-case evaluation complexity bounds for achieving approximate first- and second-order stationarity~\cite[Theorem~3.6]{cartis2020concise}.
However, these bounds do not reflect the computational cost of solving the regularized subproblem at each iteration. A range of specialized subproblem solvers for nonconvex high-order models has therefore been developed; see, for example, \cite{zhu2022quartic,cartis2025second,cartis2025cubic}.
In many regimes, the practical efficiency of AR$p$ methods is dictated not by the number of evaluations but by the difficulty (and variability) of the inner subproblem solver.

Motivated by this computational bottleneck,~\cite{silina2022unregularized} proposed an \emph{unregularized third-order Newton method} that updates $x_{k+1}$ using a strict local minimizer of the \emph{cubic} Taylor model.
A key enabling observation is that strict local minimizers of multivariate cubic polynomials can be computed via a \gls{sdp} formulation~\cite[Section~6]{ahmadi2022complexity}. This yields a unified subproblem representation that depends only on the cubic coefficients.
Under strong assumptions (e.g., strong convexity and Lipschitz continuity of the third derivative), the method achieves a cubic local convergence rate for sufficiently good initialization~\cite[Theorem~3.2]{silina2022unregularized}.

Despite its appeal, the unregularized approach is inherently local and may be ill-posed away from a minimizer: the cubic Taylor model may fail to admit any strict local minimizer, rendering the \gls{sdp} infeasible.
A standard safeguard is to add a Levenberg--Marquardt (LM) quadratic term with a sufficiently large positive coefficient, which guarantees existence of a strict local minimizer under the positive-boundary condition in \cref{lem:positive-alpha-backstop}; see also \cite[Proposition~5.1]{silina2022unregularized}.
This resolves well-posedness, but does not by itself provide a globalization mechanism or a worst-case complexity guarantee for general nonconvex objectives.

Therefore, in this work we introduce \gls{almton}, which embeds the unregularized
third-order Newton step into an \emph{adaptive} framework whose Heuristic
instantiation admits the global guarantee established below.
The key design choice is to replace the standard quartic regularization used in AR$3$ with an \emph{adaptive LM quadratic} regularization.
Crucially, this preserves a cubic model at every iteration, so that both unregularized and regularized steps can be computed through the same SDP template.
Algorithmically, \gls{almton} follows a mixed-mode strategy: it prioritizes unregularized third-order steps when the cubic model is well behaved, and increases quadratic regularization only when necessary to restore well-posedness and enforce global progress.
The numerical experiments audit basin robustness, implementation scaling, and
finite-budget trajectories.  Relative to UTON, LM globalization substantially
enlarges the set of starting points from which an audited stationary point is
returned, and on the two-dimensional benchmark of
\Cref{subsec:experiment_1:robustness} it does so in fewer recorded outer
iterations than the tested second-order comparators.  Since the per-iteration
SDP dominates the recorded time, this margin does not transfer to wall-clock
cost, and it is not interpreted as evidence of a causal advantage from
third-order information.

\subsection{Contributions and organization}

The primary contribution of this paper is the proposal of ALMTON, an adaptive
Levenberg--Marquardt third-order Newton framework that, to our knowledge,
introduces, through the Heuristic strategy, the first globally convergent
realization of the unregularized third-order Newton method of
\cite{silina2022unregularized}. The method
prioritizes the unregularized cubic Taylor step whenever this step is well
defined and sufficiently curved, and activates a quadratic LM regularization
only when needed to restore well-posedness and global progress.  If an accepted
iterate enters the stated neighborhood of a positive-definite local minimizer,
the subsequent nonterminal steps recover the unregularized third-order Newton
recursion and its cubic local rate.  For the Heuristic strategy, under the
stated assumptions and an exact local-minimizer oracle, we prove finite
termination at an \(\epsilon\)-first-order stationary point and an
\(\mathcal O(\epsilon^{-2})\) worst-case evaluation-complexity bound. We do not
claim that this improves the evaluation-complexity exponent of fully
regularized AR3 methods. Rather, the advantage of ALMTON is that each
iteration retains a tractable and uniform inner solve: both the unregularized
and LM-regularized trial steps are computed from a cubic model through the
same semidefinite programming template.

A second contribution is to isolate the computational benefit of preserving
the cubic structure. Several existing globally convergent third-order or
higher-order frameworks obtain tractable inner solves by using fourth-degree
regularized models, for example through SOS-convexification or sufficiently
large quartic regularization; see, e.g.,
\cite{ahmadi2024higher,zhu2024global,zhu2026sufficientlyregularizednonnegativequartic}. These approaches are powerful, but
their tractable subproblems are built on degree-four polynomial models. By contrast, ALMTON model remains cubic, and the same cubic-lifting SDP applies at every iteration. This yields a substantially
smaller SDP: the cubic local-minimization SDP has \(O(n^2)\) scalar variables
and positive semidefinite (PSD) constraints of size \(O(n)\), whereas a full SOS relaxation for a generic
quartic model has \(O(n^4)\)
scalar Gram variables.

On the empirical front, we complement the aggregate benchmarks with a
fixed-start trajectory and a multi-start basin audit on the deliberately
structured Hairpin Turn problem.  We report the observed paths without
interpreting them as proof of a causal geometric mechanism; see
\Cref{subsec:experiment_3:high-order}.

The remainder of this paper is organized as follows.
\Cref{sec:prelim} introduces the necessary notation and reviews the unregularized third-order Newton method.
\Cref{sec:algorithm} details the \gls{almton} framework, including the specific logic for step acceptance and regularization updates.
\Cref{sec:convergence} presents the global convergence analysis and complexity bounds. Numerical results, including robustness checks and trajectory analyses, are discussed in \Cref{sec:numerical_experiments}, followed by concluding remarks in \Cref{sec:conclusion}.

\section{Preliminaries and Notation}\label{sec:prelim}
\paragraph{Norms and tensor notation}
Throughout this paper,~$\|\cdot\|$ denotes the standard Euclidean norm on $\mathbb{R}^n$.
For a tensor $T \in \mathbb{R}^{n \times \dots \times n}$ of order $j$, we define its induced operator norm as
\begin{equation}\label{eq:tensor norm}
  \|T\|_{[j]} \;:=\; \max_{\|v_1\|=\cdots=\|v_j\|=1} \big|\,T[v_1,\ldots,v_j]\,\big|.
\end{equation}
Let~$f: \mathcal{D} \subseteq \mathbb{R}^n \to \mathbb{R}$ be a function satisfying~$f \in \mathcal{C}^p(\mathcal{D})$.
The~$p$th-order derivative of $f$ at $x$ is a symmetric $p$th-order  tensor denoted by $\nabla^p f(x)$, with entries defined by
\begin{equation}\label{eq:p-th tensor}
  \big[\nabla^p f(x)\big]_{i_1,\ldots,i_p} \;=\; \frac{\partial^p f(x)}{\partial x_{i_1}\cdots\partial x_{i_p}}.
\end{equation}
The notation $\nabla^p f(x)[s]^p$ represents the tensor acting on the vector $s \in \mathbb{R}^n$ repeated $p$ times, i.e., $\nabla^p f(x)[s]^p = \sum_{i_1,\dots,i_p} [\nabla^p f(x)]_{i_1,\dots,i_p} s_{i_1}\cdots s_{i_p}$.
Specifically, for $p=1$ and $p=2$, we recover the gradient vector $\nabla f(x)$ and the Hessian matrix $\nabla^2 f(x)$, respectively.

\paragraph{Taylor models}
For~$p \ge 1$, the~$p$th-order Taylor expansion of~$f$ around an iterate~$x_k$ is:
\begin{equation}\label{eq:taylor model}
  \taylor{p}{x_k}(x) \;:=\; f(x_k) + \sum_{j=1}^p \frac{1}{j!}\,\nabla^j f(x_k)[x-x_k]^j.
\end{equation}
Restricting our attention to~$p=3$ and defining the step~$s := x - x_k$, the cubic model can be written explicitly using tensor slices:
\begin{equation}\label{eq:regularized-series}
  \taylor{3}{x_k}(x) \;=\; f(x_k) + \nabla f(x_k)^\top s + \frac{1}{2}\,s^\top \nabla^2 f(x_k) s
  + \frac{1}{6}\sum_{i=1}^n s_i \cdot s^\top \big(\nabla^3 f(x_k)\big)_i s,
\end{equation}
where~$\big(\nabla^3 f(x_k)\big)_i \in \mathbb{R}^{n \times n}$ denotes the $i$-th matrix slice of the third derivative tensor.
Specifically, this slice corresponds to the Hessian of the $i$-th partial derivative:~$\big(\nabla^3 f(x_k)\big)_i = \nabla^2 (\partial_i f)(x_k)$.
This representation aligns with the general form of a multivariate cubic polynomial introduced in \Cref{eq:cubic-poly}, where the coefficients are identified as~$c = f(x_k)$,~$b = \nabla f(x_k)$,~$Q = \nabla^2 f(x_k)$, and~$H_i = \big(\nabla^3 f(x_k)\big)_i$.

\paragraph{Cubic polynomials and subproblem formulation}
Generalizing the specific Taylor model structure, we consider the minimization of a generic multivariate cubic polynomial $\psi: \mathbb{R}^n \rightarrow \mathbb{R}$. Consistent with the tensor slice notation introduced in \Cref{eq:regularized-series}, $\psi$ is parameterized by a set of symmetric matrices~$\{H_i\}_{i=1}^n \subset \mathbb{S}^n$ (representing the third-order interaction), a symmetric matrix~$Q \in \mathbb{S}^n$, a vector~$b \in \mathbb{R}^n$, and a scalar~$c \in \mathbb{R}$. We assume the slice-compatibility identities $(H_i)_{j\ell}=(H_j)_{i\ell}=(H_\ell)_{ij}$, so that the $H_i$ are the slices of one fully symmetric third-order tensor.
\begin{equation}\label{eq:cubic-poly}
  \psi(x) \;=\; \frac{1}{6} \sum_{i=1}^n x_i \cdot x^{\top} H_i x \;+\; \frac{1}{2} x^{\top} Q x \;+\; b^{\top} x \;+\; c.
\end{equation}
The gradient and Hessian of $\psi$ then satisfy:
\begin{align}
  \nabla \psi(x) &= \frac{1}{2}\sum_{i=1}^n x_i H_i x \;+\; Qx \;+\; b, \\
  \nabla^2 \psi(x)&=\sum_{i=1}^n x_i\,H_i\;+\;Q, \qquad
  \nabla^3 \psi(x)=\{H_i\}_{i=1}^n. \label{eq:grad-hess-slices}
\end{align}

A multivariate cubic polynomial~$\psi: \mathbb{R}^n \rightarrow \mathbb{R}$ has either no local minimizer, exactly one strict local minimizer, or infinitely many non-strict local minimizers.
Moreover,~$\bar{x}$ is a strict local minimizer iff~$\nabla\psi(\bar x)=0$ and $\nabla^2\psi(\bar x)\succ0$~\cite[Theorem 3.1, Corollary 3.4]{ahmadi2022complexity}.
If a strict local minimizer exists, it can be found by solving the following \gls{sdp}~\cite[Section 6.3.2]{ahmadi2022complexity}:
\begin{equation}\label{eq:sdp}
  \begin{aligned}
    \inf_{X\in\mathbb{S}^{n},\, x,v\in\mathbb{R}^n,\, y\in\mathbb{R}} \quad & \tfrac{1}{2}\,\mathrm{Tr}(QX) + b^\top x + \tfrac{1}{2}y \\
    \text{s.t.}\quad & \tfrac{1}{2}\,\mathrm{Tr}(H_i X) + e_i^\top Qx + b_i = 0,\;\; i=1,\ldots,n, \\
    & v = \sum_{i=1}^n \mathrm{Tr}(H_i X)\,e_i + Qx, \\
    & \quad
    \begin{bmatrix}
      \sum_{i=1}^n x_i H_i + Q & v \\
      v^\top & y
    \end{bmatrix} \succeq 0, \
    \begin{bmatrix} X & x \\ x^\top & 1
    \end{bmatrix} \succeq 0,
  \end{aligned}
\end{equation}
where~$\mathrm{Tr}$ denotes the trace and $\{e_i\}$ are the standard basis vectors.
This \gls{sdp} is strictly feasible when~$\psi$ has a unique strict local minimizer~\cite[Theorem 3.3]{silina2022unregularized}, and can be solved to arbitrary accuracy in polynomial time~\cite{vandenberghe1996semidefinite}.

\subsection{Higher-order Optimization Algorithms}\label{subsec:high-order-algorithm}
\paragraph{Unregularized third-order Newton and LM regularization}
Let~$\Phi^{3}_{f,x_k}(\cdot)$ denote the cubic Taylor model at $x_k$ (cf. \Cref{eq:regularized-series}).
The Unregularized Third-Order Newton method~\cite[Algorithm~1]{silina2022unregularized} generates iterates~$\{x_k\}$ analogously to classical Newton's method: it terminates when~$\|\nabla f(x_k)\|\le \epsilon,$
and otherwise computes a higher-order Newton step by taking~$x_{k+1}$ as a \emph{strict local minimizer} of the cubic model,
\begin{equation*}
    \text{choose }x_{k+1}\text{ as a strict local minimizer of }\Phi^{3}_{f,x_k}.
\end{equation*}
When such a minimizer exists, it can be found via the unified \gls{sdp} formulation in \Cref{eq:sdp}.
Far from a minimizer of~$f$, however,~$\Phi^{3}_{f,x_k}$ may admit no (strict) local minimizer.

As noted in the discussion on cubic polynomials (cf. \Cref{eq:cubic-poly}), the unregularized model~$\Phi^{3}_{f,x_k}$ is not guaranteed to possess a strict local minimizer; it may have no local minimizers or infinitely many non-strict ones, rendering the update step ill-defined.
A standard safeguard is to augment the model with a Levenberg--Marquardt (LM) regularization term and minimize instead:
\begin{equation*}
  m_{f,x_k}(x;\sigma) \;:=\; \Phi^{3}_{f,x_k}(x) \;+\; \sigma\,\|x-x_k\|^2.
\end{equation*}
For sufficiently large~$\sigma$,~$m_{f,x_k}(\cdot;\sigma)$ has a local minimizer~\cite[Prop.~5.1]{silina2022unregularized}.
A computable \emph{sufficient} threshold is
\begin{equation}\label{eq:alpha_LM}
  \begin{aligned}
    & g :=\big(|\partial_1 f(x_k)|,\ldots,|\partial_n f(x_k)|\big)^{\top}, \
    h :=\big(\|(\nabla^3 f(x_k))_1\|_2,\ldots,\|(\nabla^3 f(x_k))_n\|_2\big)^{\top}, \\
    & \alpha_{LM}(x_k)\;:=\;\sqrt{\frac{3}{2}\,\big(\|g\|\,\|h\|+g^\top h\big)}\;-\;\min\{0,\lambda_{\min}(\nabla^2 f(x_k))\}.
  \end{aligned}
\end{equation}
The following positive-boundary refinement of
\cite[Proposition~5.1]{silina2022unregularized} will be used.
\begin{lemma}[Positive LM backstop]\label{lem:positive-alpha-backstop}
For every $\sigma\ge\alpha_{LM}(x_k)$ with $\sigma>0$,
the regularized cubic model has a strict local minimizer.
\end{lemma}
\begin{proof}
  A self-contained proof is given in Appendix~\ref{app:positive-alpha-backstop}.
\end{proof}

This refinement is needed in the finite
model-phase argument of \cref{prop:model-phase}, where invoking the backstop
must produce a bona fide strict local minimizer even when $\alpha_{LM}=0$.
The condition $\sigma>0$ cannot be omitted: for $f(t)=t+t^4$ at $x_k=0$,
one has $\alpha_{LM}=0$, whereas the unregularized cubic Taylor model is the
linear function $t$ and has no local minimizer.  The condition is automatic in
both strategies below: Simple assigns $\max\{1,\alpha_{LM}(x_k)\}>0$ when the
outer parameter is zero and thereafter multiplies positive values by $\gamma$,
whereas every Heuristic backstop update contains
$\gamma\max\{1,\tilde\sigma\}>0$.  Hence no parameter-update rule needs to be
modified.
We use~$\alpha_{LM}$ as a \emph{backstop} in the analysis; it need not be tight.

\begin{remark}[Unified subproblem solve]
  In both the unregularized case and the case regularized by LM, the model
  remains \emph{cubic}.
  Hence, whenever a (strict) local minimizer exists, we can solve the subproblem with the same \gls{sdp} template \eqref{eq:sdp}.
  This unification is a key practical advantage carried throughout the paper.
\end{remark}

\paragraph{AR\texorpdfstring{$p$}{p} models}

The AR$p$ framework augments~$\taylor{p}{\bar x}$ with a~$(p{+}1)$th-degree term:
\begin{equation*}
  \widetilde{m}^{p}_{f,\bar x}(x,\tilde\sigma)=f(\bar x)+\sum_{j=1}^p \frac{1}{j!}\nabla^j f(\bar x)[x-\bar x]^j + \frac{\tilde\sigma}{p+1}\|x-\bar x\|^{p+1}.
\end{equation*}
For sufficiently large~$\tilde\sigma$, the model~$\widetilde{m}^{p}_{f,\bar x}$ is bounded below (and coercive), guaranteeing that its global minimizer exists.
The regularization parameter is adapted by a ratio test to ensure sufficient descent~\cite{cartis2011adaptive,cartis2011adaptive2,birgin2017worst,cartis2020concise,cartis2024efficient}.
In the special case~$p=3$ (AR$3$), the added term is quartic.

\paragraph{Trade-offs motivating our approach}
The AR$p$ philosophy is to guarantee robustness by \emph{regularizing the model}: add a ($p+1$)th-degree term with a sufficiently large coefficient $\tilde{\sigma}$, obtain a coercive surrogate whose global minimizer is guaranteed to exist, then select steps via a ratio test. Under standard assumptions, AR$p$ methods admit worst-case evaluation complexity bounds for achieving approximate first- and second-order stationarity~\cite[Theorem~3.6]{cartis2020concise}.
A sequence of specialized solvers for the nonconvex high-order regularization subproblems has therefore been developed, including quartic, quadratic--quartic, and cubic--quartic subproblem approaches~\cite{zhu2022quartic,cartis2025second,cartis2025cubic}.
This works well conceptually, but in practice, it creates two frictions that are central to our design choices.

First, for~$p=3$ (AR$3$), the quartic regularization ensures coercivity, yet the resulting subproblem may remain nonconvex for moderate~$\tilde{\sigma}$, and there is no single ``one-size-fits-all'' solver that covers the spectrum of problems and accuracies practitioners need.
Implementations often rely on problem-specific routines or black-box smooth minimizers with their own heuristics.
For the standard ARC ($p=2$) subproblem, efficient iterative methods, such as the factorization-based MCM algorithm~\cite[Algorithm 6.1]{cartis2011adaptive}, are often used to find the global minimizer.
For the AR$3$ ($p=3$) subproblem, common approaches within the ``AR$3$-BGMS''~\cite{cartis2024efficient} variant include specialized algorithms such as Gencan~\cite[Procedure 4.2]{cartis2024efficient}, formulations based on \gls{sos} programming~\cite{zhu2024global,ahmadi2024higher,lasserre2001global,zhu2025globaloptimalitycharacterizationsalgorithms,zhu2026sufficientlyregularizednonnegativequartic}, and the use of AR$2$-based~\cite[Procedure 3.2, 4.1]{cartis2024efficient} methods as local solvers.
By contrast, with an LM term the model remains \emph{cubic} for all~$\sigma$.
Whenever a (strict) local minimizer exists, whether unregularized or with LM,
we can invoke the \emph{same} \gls{sdp} template \eqref{eq:sdp}, with shared
certificate checks and a consistent stopping interface.  The reported runs do
not use solver warm starts.

Second, large regularization can erode local fidelity.
As~$\tilde\sigma$ grows, the regularizer increasingly dominates the Taylor terms, making predicted decrease less representative of the true objective's behavior, especially far from a solution.
Empirically this can manifest as conservative steps, more unsuccessful iterations, and weaker correlation in the ratio test.
In the LM-regularized cubic setting the model's polynomial degree and
representation remain fixed.  This is a design motivation rather than a
guarantee of better calibration or lower cost; those effects are assessed
empirically below.

The practical upshot is a \emph{mixed-mode} strategy: prefer an unregularized third-order step when the cubic model admits a strict local minimizer and curvature is adequate; otherwise, increase $\sigma$ to cross a computable backstop (e.g., $\alpha_{LM}$) that ensures existence of a local minimizer.
This retains the desirable fast local behavior when available, while providing a principled globalization mechanism---without switching subproblem classes across iterations.
Our focus on LM-regularized cubics is therefore not a rejection of AR$p$ theory, but a complementary design that prioritizes a unified subproblem solver and more stable model--objective fidelity.

\subsection{Standing assumptions and basic bounds}
To facilitate the analysis, we adopt the following standard assumptions.

\begin{assumption}[Smoothness and lower boundedness]\label{assump:1}
  Let~$p\ge3$ be fixed.
  We consider functions~$f:\mathbb{R}^n\to\mathbb{R}$ satisfying:
  \begin{enumerate}
    \item $f\in\mathcal{C}^p(\mathbb{R}^n)$, i.e., $f$ is $p$-times \emph{continuously} differentiable;
    \item $f$ is bounded below by $f_{\mathrm{low}}$;
    \item the $p$th-order derivative is Lipschitz (in operator norm): there exists $L\ge0$ such that, for all $x,y\in\mathbb{R}^n$,
      \begin{equation}\label{eq:lipschitz}
        \big\|\nabla^p f(x)-\nabla^p f(y)\big\|_{[p]} \;\le\; (p-1)!\,L\,\|x-y\|.
      \end{equation}
  \end{enumerate}
\end{assumption}

\cref{assump:1} aligns with the standing conditions in high-order methods (e.g., AR$p$) and yields the usual Taylor remainder bounds~\cite{birgin2017worst,cartis2011adaptive,cartis2020concise,zhu2024global,cartis2025second}.

\begin{assumption}[Uniform bounds along the iterates \cite{zhu2024global}]
  \label{assump:2}
  Along the sequence $\{x_k\}$ produced by the algorithm, there exist finite constants $\{\Lambda_j\}_{j=1}^p$ such that
  \begin{equation*}
    \big\|\nabla^j f(x_k)\big\|_{[j]} \;\le\; \Lambda_j, \qquad \forall\,k\ge0,\; j=1,\ldots,p.
  \end{equation*}
\end{assumption}

\begin{remark}
  \cref{assump:2} is \emph{not} a global requirement; it only constrains the derivatives \emph{along the iterates}.
  It is implied by \cref{assump:1} whenever the iterates remain in a bounded set---which holds here, since $\{f(x_k)\}$ is nonincreasing (\Cref{theorem:monotonicity}) and bounded below by $f_{\mathrm{low}}$, so the iterates stay in a sublevel set of $f$ (bounded when $f$ has bounded sublevel sets). The complexity constants enter only through $\{\Lambda_j\}$ via $\alpha_{\max}$ in \cref{eq:alpha_max}; we make this dependence explicit there.
\end{remark}

We now recall the Taylor remainder bounds that will be used repeatedly.

\begin{lemma}[Taylor model bounds]\label{lemma:taylor_model}
  Under \cref{assump:1}, for all $x,x_k\in\mathbb{R}^n$,
  \begin{equation}\label{eq:upper-bound-gradient}
    f(x) \le \taylor{p}{x_k}(x) \;+\; \frac{L}{p}\,\|x-x_k\|^{p+1}, \ \big\|\nabla f(x) - \nabla \taylor{p}{x_k}(x)\big\| \le L\,\|x-x_k\|^{p}.
  \end{equation}
\end{lemma}
\begin{proof}
  See~\cite[Eq.\ (2.6)--(2.7)]{birgin2017worst}.
\end{proof}

\begin{corollary}[Specialization to $p=3$]\label{cor:p3}
  With $p=3$ in \cref{lemma:taylor_model}, for $s:=x-x_k$,
  \begin{equation*}
    f(x) \;\le\; \taylor{3}{x_k}(x) \;+\; \frac{L}{3}\,\|s\|^{4},
    \qquad
    \big\|\nabla f(x)-\nabla \taylor{3}{x_k}(x)\big\| \;\le\; L\,\|s\|^{3}.
  \end{equation*}
\end{corollary}

\section{\gls{almton}}\label{sec:algorithm}
Building on \Cref{sec:prelim}, we now specialize to $p=3$ and give a complete description of our method.
The key design choice is to \emph{prefer} unregularized third-order steps whenever the cubic Taylor model admits a strict local minimizer with adequate curvature, and to \emph{fall back} to an LM-regularized cubic otherwise.
This preserves fast local behavior while providing a principled globalization mechanism, all within a unified cubic subproblem (solved by the \gls{sdp} in \cref{eq:sdp}).

\subsection*{Cubic model with LM regularization}
For any iterate~$x_k$ and regularization~$\sigma\ge0$, we define the cubic model
\begin{equation}\label{eq:taylor model with regularization}
  m_{f,x_k}(x;\sigma)\;:=\;\taylor{3}{x_k}(x)+\sigma\,\|x-x_k\|^2,
\end{equation}
so that~$m_{f,x_k}(x_k;\sigma)=f(x_k)$ and
\begin{equation*}
  \nabla m_{f,x_k}(x;\sigma)\;=\;\nabla \taylor{3}{x_k}(x)\;+\;2\sigma(x-x_k).
\end{equation*}
When~$\sigma$ is sufficiently large,~$m_{f,x_k}(\cdot;\sigma)$ admits a local minimizer (see \Cref{eq:alpha_LM}); conversely, when~$\sigma=0$, the model may be unbounded below or lack a local minimizer.
Our framework exploits this structure by prioritizing the unregularized case ($\sigma=0$) to exploit favorable local geometry, and increasing~$\sigma$ only when necessitated by the model's behavior.

We distinguish the two Hessians that enter the analysis.  For $s=x-x_k$, let
\begin{equation}\label{eq:A-M-definitions}
  A_k(\sigma):=\nabla^2f(x_k)+2\sigma I_n,
  \qquad
  M_k(s;\sigma):=A_k(\sigma)+\nabla^3f(x_k)[s].
\end{equation}
Thus $A_k(\sigma)$ is the model Hessian at the base point $x_k$, whereas
$M_k(s;\sigma)$ is the model Hessian at the trial point $x_k+s$.  Neither
matrix is the true trial Hessian
$\nabla^2f(x_k+s)$.  The model phase tests $A_k(\sigma)\succeq cI_n$; local
minimality supplies $M_k(s;\sigma)\succeq0$.

\begin{remark}[Cubic SDP versus quartic SOS lifting]
The difference between preserving a cubic model and passing to a quartic
regularized model is substantial at the level of the lifted SDP. The cubic
SDP used in ALMTON has
\[
\frac{n(n+1)}2+2n+1=O(n^2)
\]
scalar variables, together with two positive semidefinite (PSD) constraints of size \(n+1\). In
contrast, a full SOS relaxation for a generic fourth-degree model uses a
Gram matrix indexed by all monomials of degree at most two. This monomial
basis has size
$
N_2=\binom{n+2}{2}=\frac{(n+1)(n+2)}2=O(n^2),
$
so the Gram matrix contains \(N_2(N_2+1)/2=O(n^4)\) scalar variables.

For example, when \(n=10\), the cubic SDP has about \(76\) scalar variables
and two PSD blocks of size \(11\), whereas the quartic SOS relaxation has
\(2211\) scalar Gram variables and a PSD block of size \(66\). When \(n=20\),
the corresponding numbers are about \(251\) versus \(26796\) scalar variables,
with PSD block sizes \(21\) versus \(231\). Thus replacing the quadratic LM
term by a quartic regularizer would change the subproblem from an
\(O(n^2)\)-variable, \(O(n)\)-PSD-block SDP to an \(O(n^4)\)-variable,
\(O(n^2)\)-PSD-block SOS relaxation.
\end{remark}

\subsection*{The ALMTON Framework}

We first present the general control flow in \Cref{alg:ALMTON_Meta}.
This meta-algorithm encapsulates the initialization, termination criteria, and the acceptance ratio test.
Crucially, it delegates the specific mechanisms for computing the trial step~$s_k$ and updating the regularization parameter~$\sigma_{k+1}$ to a specific \emph{strategy} (denoted by $\mathcal{U}$).

\begin{algorithm}[h!]
  \begin{algorithmic}[1]
    \caption{Adaptive Levenberg--Marquardt Third-Order Newton Method (ALMTON): General Framework}
    \label{alg:ALMTON_Meta}
    \STATE \textbf{Step 0: Initialization.}
    Given $f$, $x_0$, tolerance $\epsilon>0$, and constants $c>0$, $l\in(0,c/6]$ (e.g., $l=c/10$), $\eta\in(0,1)$, $\gamma>1$, set $\sigma_0:=0$, compute $f(x_0)$, and set $k:=0$.
    \smallskip
    \STATE \textbf{Step 1: Termination.}
    Compute $\nabla f(x_k)$. If $\|\nabla f(x_k)\|\le \epsilon$, terminate with $x_\epsilon:=x_k$. Otherwise, compute $\nabla^2 f(x_k)$, $\nabla^3 f(x_k)$, and the threshold $\alpha_{LM}(x_k)$ from \cref{eq:alpha_LM}.
    \smallskip
    \STATE \textbf{Step 2: Step Calculation (Model Phase).}
    Execute the chosen strategy (e.g., \textbf{Simple} in \cref{alg:Strategy_Simple} or \textbf{Heuristic} in \cref{alg:Strategy_Heuristic}) to obtain the trial step $s_k$ and phase regularization $\tilde\sigma_k$.
    \smallskip
    \STATE \textbf{Step 3: Acceptance of the trial point.}
    \IF{$s_k=0$}
    \STATE Set $\rho_k:=-\infty$ and $x_{k+1}:=x_k$, and proceed directly to Step~4 without a trial-point objective evaluation. \COMMENT{Invalid model step}
    \ENDIF\space
    Evaluate $f(\bar x)$ and define the ratio
    \begin{equation}\label{eq:rho}
      \rho_k =
      \begin{cases}
        \dfrac{f(x_k) - f(\bar x)}{\,l\,\|s_k\|^{2}\,}, & \text{if } \tilde\sigma_k=0, \\[2mm]
        \dfrac{f(x_k) - f(\bar x)}{\,f(x_k) -  m_{f,x_k}(\bar x;\tilde\sigma_k)\,}, & \text{if } \tilde\sigma_k>0.
      \end{cases}
    \end{equation}
    If $\rho_k \ge \eta$ (a \textbf{successful} iteration), set $x_{k+1}:=\bar x$; otherwise set $x_{k+1}:=x_k$.
    \smallskip
    \STATE \textbf{Step 4: Regularization parameter update.}
    \begin{equation*}
      \sigma_{k+1} =
      \begin{cases}
        0 & \text{if } \rho_k \ge \eta \quad \text{(\textbf{success})}, \\
        \mathcal{U}(\sigma_k, \tilde\sigma_k) & \text{if } \rho_k < \eta \quad \text{(\textbf{failure})},
      \end{cases}
    \end{equation*}
    where $\mathcal{U}(\cdot)$ denotes the strategy-specific update rule defined in \cref{alg:Strategy_Simple} or \cref{alg:Strategy_Heuristic}. Increment $k\leftarrow k+1$ and go to Step~1.
  \end{algorithmic}
\end{algorithm}

At every nonterminal iteration, $\nabla f(x_k)\ne 0$ and
$\nabla_x m_{f,x_k}(x_k;\tilde\sigma)=\nabla f(x_k)$. Hence an exact local minimizer of the model cannot have $s_k=0$, so the zero step unambiguously identifies an invalid Simple trial.

\subsection*{Instantiating the Strategy}

\cref{alg:ALMTON_Meta} serves as a template for two distinct strategies:
a \textit{Simple} version and a \textit{Heuristic} version.
The global convergence and worst-case evaluation-complexity analysis below
is for the Heuristic strategy under the stated exact local-minimizer oracle
convention.  Simple is retained as the one-logical-SDP-per-outer-iteration
variant and is compared empirically; no separate total-outer-iteration
complexity theorem for Simple is claimed here.
The practical performance of these strategies is contrasted in
\cref{sec:numerical_experiments}.

\paragraph{Strategy 1: The Simple Variant}
\cref{alg:Strategy_Simple} adopts a minimalist approach.
Upon an unsuccessful iteration, it triggers an exponential increase in~$\sigma_k$ at the outer level, and the new value is used in the next model phase.

\begin{algorithm}[h!]
  \begin{algorithmic}[1]
    \caption{Strategy: ALMTON-Simple}
    \label{alg:Strategy_Simple}
    \STATE \textbf{Mechanism: Step Calculation.}
    \STATE Set $\tilde\sigma \leftarrow \sigma_k$. Attempt to compute a local minimizer $\bar x$ of $m_{f,x_k}(\cdot;\tilde\sigma)$.
    \IF{minimizer $\bar x$ exists \textbf{and} $\lambda_{\min}(A_k(\tilde\sigma)) \ge c$}
    \STATE Set $s_k := \bar x - x_k$ and $\tilde\sigma_k := \tilde\sigma$.
    \ELSE
    \STATE Set $s_k := 0$ and $\tilde\sigma_k := \tilde\sigma$. \COMMENT{Step~3 explicitly sets $\rho_k=-\infty$}
    \ENDIF
    \smallskip
    \STATE \textbf{Mechanism: Parameter Update $\mathcal{U}(\sigma_k, \tilde\sigma_k)$.}
    \begin{equation*}
      \mathcal{U}(\sigma_k, \tilde\sigma_k) =
      \begin{cases}
        \max\{1, \alpha_{LM}(x_k)\}, & \text{if } \sigma_k = 0, \\
        \gamma \sigma_k, & \text{if } \sigma_k > 0.
      \end{cases}
    \end{equation*}
  \end{algorithmic}
\end{algorithm}

The Simple strategy constructs at most one logical \gls{sdp} subproblem per iteration and delegates regularization growth entirely to the outer loop. Multiple physical \verb|MOSEK| invocations may occur only as numerically equivalent scaling or precision attempts for that same subproblem.

\paragraph{Strategy 2: The Heuristic Variant}
The Heuristic strategy is motivated by the high computational cost of solving the \gls{sdp} subproblem.
Instead of rigidly rejecting any step failing the curvature condition, it uses an inner correction loop that actively searches for a valid regularization parameter~$\tilde\sigma_k$ guaranteeing that the model is well-posed \emph{before} evaluating the objective function.
The spectral correction uses $\lambda_k(\tilde\sigma):=\lambda_{\min}(A_k(\tilde\sigma))$ and the term $\tilde\sigma+(c-\lambda_k(\tilde\sigma))_+$; the full correction is conservative because increasing $\tilde\sigma$ by $\delta$ shifts $A_k$ by $2\delta I_n$.

\begin{algorithm}[h!]
  \begin{algorithmic}[1]
    \caption{Strategy: ALMTON-Heuristic}
    \label{alg:Strategy_Heuristic}
    \STATE \textbf{Mechanism: Step Calculation.}
    \STATE Set $\tilde\sigma \leftarrow \sigma_k$.
    \REPEAT
    \STATE Attempt to compute a local minimizer $\bar x$ of $m_{f,x_k}(\cdot;\tilde\sigma)$.
    \IF{no such minimizer exists}
    \STATE Increase $\tilde\sigma \leftarrow \max\!\big\{\alpha_{LM}(x_k),\,\gamma\max\{1,\tilde\sigma\}\big\}$ and \textbf{continue}. \COMMENT{No spectral correction}
    \ENDIF
    \STATE Let $\lambda_k(\tilde\sigma) := \lambda_{\min}(A_k(\tilde\sigma))$.
    \IF{$\lambda_k(\tilde\sigma) \ge c$ \textbf{and} $m_{f,x_k}(\bar x;\tilde\sigma)\le f(x_k)$}
    \STATE \textbf{break} \COMMENT{Valid step found}
    \ENDIF
    \STATE Increase $\tilde\sigma \leftarrow \max\!\big\{\alpha_{LM}(x_k),\,\gamma\max\{1,\tilde\sigma\},\,\tilde\sigma+(c-\lambda_k(\tilde\sigma))_+\big\}$.
    \UNTIL{valid step found}
    \STATE Set $s_k := \bar x - x_k$ and $\tilde\sigma_k := \tilde\sigma$.
    \smallskip
    \STATE \textbf{Mechanism: Parameter Update $\mathcal{U}(\sigma_k, \tilde\sigma_k)$.}
    \begin{equation*}
      \mathcal{U}(\sigma_k, \tilde\sigma_k) =
      \max\!\big\{\alpha_{LM}(x_k), \, \gamma \max\{1, \tilde\sigma_k\}\big\}.
    \end{equation*}
  \end{algorithmic}
\end{algorithm}

\begin{remark}
For the Heuristic strategy (\Cref{alg:Strategy_Heuristic}), the inner loop always terminates with a trial point~$\bar x$ satisfying the curvature and model-value conditions.
However, the monotone increase of~$\tilde\sigma$ can lead to
\emph{over-regularization} and smaller steps.  Its reported phase count is not
a normalized measure of work, and the experiments therefore also report SDP
requests and wall time.
\end{remark}

\paragraph{Phase notation}
For clarity, we write~$\tilde\sigma_k$ for the value of the regularization at the end of the model phase (i.e., the value with which~$\bar x$ and~$s_k$ are computed).

\paragraph{First- and second-order conditions at the trial point}
Let $s_k:=\bar x-x_k$.  Since $\bar x$ is a local minimizer of the phase
model, the necessary conditions are
\begin{align}
  \nabla_xm_{f,x_k}(\bar x;\tilde\sigma_k)
  &=\nabla f(x_k)+A_k(\tilde\sigma_k)s_k
    +\tfrac12\nabla^3f(x_k)[s_k,s_k]=0,
    \label{eq:kkt-stationarity}\\
  \nabla_x^2m_{f,x_k}(\bar x;\tilde\sigma_k)
  &=M_k(s_k;\tilde\sigma_k)\succeq0.
    \label{eq:kkt-curvature}
\end{align}

\begin{lemma}[Exact regularized model-decrease identity]\label{lem:exact-identity}
Suppose $s_k$ satisfies \eqref{eq:kkt-stationarity}, and set
$A_k:=A_k(\tilde\sigma_k)$ and $M_k:=M_k(s_k;\tilde\sigma_k)$.  Then
\begin{equation}\label{eq:exact-identity}
  D_k:=f(x_k)-m_{f,x_k}(\bar x;\tilde\sigma_k)
  =s_k^\top\!\left(\tfrac16A_k+\tfrac13M_k\right)s_k.
\end{equation}
Consequently, if $s_k\ne0$, $A_k\succeq cI_n$, and $M_k\succeq0$, then
\begin{equation}\label{eq:lb-unreg}
  D_k\ge\frac{c}{6}\|s_k\|^2>0.
\end{equation}
\end{lemma}
\begin{proof}
Write $a:=s_k^\top A_ks_k$ and
$t:=\nabla^3f(x_k)[s_k,s_k,s_k]$.  Multiplying
\eqref{eq:kkt-stationarity} by $s_k^\top$ gives
$\nabla f(x_k)^\top s_k+a+t/2=0$.  Hence
\[
  D_k=-\nabla f(x_k)^\top s_k-\tfrac12a-\tfrac16t
     =\tfrac16a+\tfrac13(a+t),
\]
which is \eqref{eq:exact-identity}, because
$a+t=s_k^\top M_ks_k$.  The lower bound follows immediately.
\end{proof}

\subsection*{Model-phase invariants and basic descent}

\begin{proposition}[Finite model phase and its invariants]\label{prop:model-phase}
Under the exact local-minimizer oracle convention that an attempt returns a
local minimizer whenever one exists and declares failure only when none
exists, Step~2 of the Heuristic strategy terminates after finitely many inner
updates at every nonterminal iteration and returns
$(\bar x,\tilde\sigma_k)$ such that, with $s_k:=\bar x-x_k$,
\begin{enumerate}\itemsep0.2em
  \item[\textup{(i)}] $\bar x$ is a strict local minimizer of
    $m_{f,x_k}(\cdot;\tilde\sigma_k)$;
  \item[\textup{(ii)}] $A_k(\tilde\sigma_k)\succeq cI_n$;
  \item[\textup{(iii)}]
    $f(x_k)-m_{f,x_k}(\bar x;\tilde\sigma_k)
      \ge(c/6)\|s_k\|^2>0$.
\end{enumerate}
\end{proposition}
\begin{proof}
Whenever the phase does not exit, its update satisfies
$\tilde\sigma^+\ge\gamma\max\{1,\tilde\sigma\}>\tilde\sigma$.
Thus, after finitely many updates, $\tilde\sigma$ is positive, is at least
$\alpha_{LM}(x_k)$, and is large enough that
$A_k(\tilde\sigma)\succeq cI_n$.  By
\cref{lem:positive-alpha-backstop}, a strict local minimizer then exists.
At any local minimizer, \eqref{eq:kkt-curvature} gives $M_k\succeq0$;
\Cref{lem:exact-identity} then makes the model-value test strict, so all exit
tests pass.  This proves finite termination and (ii)--(iii).

For completeness, the curvature floor also upgrades any local minimizer found
at an earlier candidate to a strict one.  If a nonzero $z$ belonged to
$\ker M_k$, local minimality of the shifted cubic
\[
  y\longmapsto \tfrac12y^\top M_ky
     +\tfrac16\nabla^3f(x_k)[y,y,y]
\]
would imply $\nabla^3f(x_k)[z,z,w]=0$ for every $w$: substitute
$y=tz$ first and use both signs of $t$ to get
$\nabla^3f(x_k)[z,z,z]=0$.  Next, for $y=tz+a t^2w$,
\[
  \tfrac12y^\top M_ky+\tfrac16\nabla^3f(x_k)[y,y,y]
  =\frac{t^4}{2}\!\left(a^2w^\top M_kw
       +a\nabla^3f(x_k)[z,z,w]\right)+O(t^5).
\]
For each fixed $a$, local minimality requires the coefficient in parentheses
to be nonnegative.  If $\nabla^3f(x_k)[z,z,w]\ne0$, choosing $a$ with the
opposite sign and sufficiently small magnitude makes it negative, a
contradiction.  Taking $w=s_k$ therefore gives
$z^\top A_kz=z^\top M_kz-\nabla^3f(x_k)[z,z,s_k]=0$, contradicting
$A_k\succeq cI_n$.  Hence $M_k\succ0$, and the local minimizer is strict.
Finally, a nonterminal iteration has $\nabla f(x_k)\ne0$, so
\eqref{eq:kkt-stationarity} excludes $s_k=0$.  This proves (i) and the strict
inequality in (iii).
\end{proof}

\begin{corollary}[Uniform quadratic model decrease]\label{cor:unreg-l}
Every valid phase step satisfying the curvature floor obeys
\[
  f(x_k)-m_{f,x_k}(\bar x;\tilde\sigma_k)
  \ge\frac{c}{6}\|s_k\|^2\ge l\|s_k\|^2.
\]
In particular, this holds for the unregularized phase $\tilde\sigma_k=0$.
\end{corollary}
\begin{proof}
Apply \Cref{lem:exact-identity}, \eqref{eq:kkt-curvature}, and
$A_k(\tilde\sigma_k)\succeq cI_n$, and use $l\le c/6$.
\end{proof}

\subsection*{Acceptance ratio: justification and calibration}
When $\tilde\sigma_k>0$, the denominator~$f(x_k)-m_{f,x_k}(\bar x;\tilde\sigma_k)$ in \cref{eq:rho} is the standard trust-region/ARC predicted decrease.
For both regularized and unregularized valid steps, \Cref{cor:unreg-l} gives
$f(x_k)-m_{f,x_k}(\bar x;\tilde\sigma_k)\ge(c/6)\|s_k\|^2$.
We therefore fix $l\in(0,c/6]$ (e.g., $l=c/10$).  In the unregularized
ratio, $l\|s_k\|^2$ is a conservative acceptance scale; it is not asserted to
equal the full predicted decrease.

\begin{lemma}[Descent per successful iteration]\label{lem:descent}
  Let $k$ be such that $\rho_k\ge \eta$ and $x_{k+1}=\bar x$.
  Then:
\begin{enumerate}\itemsep0.2em
    \item[\textup{(a)}] If $\tilde\sigma_k=0$, then $f(x_{k+1}) \le f(x_k) - \eta\,l\,\|s_k\|^2$.
    \item[\textup{(b)}] If $\tilde\sigma_k>0$, then $f(x_{k+1}) \le f(x_k) - \eta\big(f(x_k)-m_{f,x_k}(\bar x;\tilde\sigma_k)\big)\le f(x_k)-\eta l\|s_k\|^2$.
  \end{enumerate}
  In particular, $\{f(x_k)\}$ is nonincreasing and strictly decreases on every successful iteration.
\end{lemma}

\begin{proof}
  By the definition of $\rho_k$ in \cref{eq:rho} and $\rho_k\ge\eta$,
  \[
    f(x_k)-f(x_{k+1})
    =
    \begin{cases}
      \rho_k\,l\,\|s_k\|^2, & \tilde\sigma_k=0,\\[1mm]
      \rho_k\,\big(f(x_k)-m_{f,x_k}(\bar x;\tilde\sigma_k)\big), & \tilde\sigma_k>0,
    \end{cases}
  \]
  which yields (a), and the first inequality in (b).  The second inequality in
  (b) follows from \Cref{cor:unreg-l}.  Thus every successful iteration obeys
  \begin{equation}\label{eq:unified-success-decrease}
    f(x_k)-f(x_{k+1})\ge\eta l\|s_k\|^2>0.
  \end{equation}
  Monotonicity follows, since a failed iteration leaves the iterate unchanged.
\end{proof}

\begin{remark}[What counts as ``successful'' and how feasibility is checked]
An iteration is \emph{successful} only if Step~3 is reached with a bona fide trial point $\bar x$ (hence $s_k=\bar x-x_k\neq 0$ unless already converged) and the mixed ratio test \cref{eq:rho} yields $\rho_k\ge\eta$.
The convergence theory uses an exact local-minimizer oracle: the returned point
satisfies \eqref{eq:kkt-stationarity}--\eqref{eq:kkt-curvature} exactly, and
the oracle declares failure if and only if no local minimizer exists.  The
numerical implementation approximates this oracle with the \gls{sdp} in
\cref{eq:sdp}; a separate inexact-oracle analysis would be needed to propagate
solver residuals through the stated constants.
No ratio~$\rho_k$ is formed until the invariants in \Cref{prop:model-phase} are satisfied (for the Heuristic strategy).
\end{remark}

\begin{remark}[On the roles of $c$, $l$, and $\alpha_{\!LM}$]
The curvature floor $c>0$ is imposed on the base model Hessian
$A_k(\tilde\sigma_k)$.  Together with local minimality, it yields the uniform
model-decrease bound in \Cref{cor:unreg-l}; hence any fixed
$l\in(0,c/6]$ is valid without a second spectral condition.  The LM backstop
$\alpha_{LM}(x_k)$ guarantees strict local-minimizer existence when combined
with the positive safeguard in \cref{lem:positive-alpha-backstop}.
\end{remark}

\begin{remark}[Resetting to unregularized after success]
  After any successful iteration ($\rho_k\ge\eta$), we set $\sigma_{k+1}=0$ to re-test the unregularized cubic at the next iterate.
  This is how ALMTON recovers fast local behavior whenever the cubic admits a
  well-conditioned local minimizer.  If it is ill posed, Heuristic increases
  $\sigma$ within the model phase until the backstops in
  \cref{prop:model-phase} hold, whereas Simple updates $\sigma$ after a failed
  outer iteration.
\end{remark}

\begin{figure}[tb]
  \centering
  \hspace*{-0.1\textwidth}%
  \resizebox{1.1\textwidth}{!}{
    \usetikzlibrary{shapes.geometric, arrows.meta, positioning, calc, shadows}

\definecolor{almton_blue}{RGB}{220, 230, 245}
\definecolor{almton_red}{RGB}{250, 220, 220}
\definecolor{almton_green}{RGB}{220, 245, 220}
\definecolor{almton_brown}{RGB}{245, 235, 210}
\definecolor{almton_purple}{RGB}{230, 220, 245}
\definecolor{arrow_blue}{RGB}{16, 39, 130}

\begin{tikzpicture}[
    node distance=1.8cm and 1.2cm,
    font=\small\sffamily,
    process/.style={
        rectangle,
        minimum width=3cm,
        minimum height=1.8cm,
        text centered,
        draw=black!60,
        thick,
        rounded corners=5pt,
        drop shadow,
        align=center
    },
    decision/.style={
        diamond,
        aspect=2,
        minimum width=2.5cm,
        minimum height=1.2cm,
        text centered,
        draw=black!60,
        thick,
        fill=white,
        drop shadow
    },
    line/.style={
        draw,
        arrow_blue,
        line width=1.2pt,
        -{Latex[length=3mm]}
    }
]

    \node (step0) [process, fill=almton_blue] {
        \textbf{Step 0: Initialization}\\
        Given $f, x_0, \epsilon$\\
        Set $\sigma_0=0, k=0$
    };

    \node (step1) [process, fill=almton_red, right=of step0] {
        \textbf{Step 1: Termination}\\
        Check $\|\nabla f(x_k)\| \le \epsilon$\\
        Compute $\nabla f, \nabla^2 f, \nabla^3 f$
    };

    \node (terminate) [process, fill=gray!20, above=of step1, minimum height=1cm, yshift=-0.5cm] {
        \textbf{Terminate}\\
        Return $x_\epsilon = x_k$
    };

    \node (step2) [process, fill=almton_green, right=of step1] {
        \textbf{Step 2: Model Phase}\\
        Execute Strategy $\mathcal{U}$\\
        (Simple / Heuristic)\\
        \textit{Output:} $s_k, \tilde{\sigma}_k$
    };

    \node (step3) [process, fill=almton_brown, right=of step2] {
        \textbf{Step 3: Acceptance}\\
        Compute ratio $\rho_k$\\
        \textbf{If} $\rho_k \ge \eta$: $x_{k+1} \leftarrow \bar{x}$\\
        \textbf{Else}: $x_{k+1} \leftarrow x_k$
    };

    \node (step4) [process, fill=almton_purple, right=of step3] {
        \textbf{Step 4: Update $\sigma$}\\
        \textbf{If Success:} $\sigma_{k+1} \leftarrow 0$\\
        \textbf{If Failure:} $\sigma_{k+1} \leftarrow \mathcal{U}(\sigma_k, \tilde{\sigma}_k)$\\
        Set $k \leftarrow k+1$
    };

    \draw [line] (step0) -- (step1);

    \draw [line] (step1.north) -- node[midway, left, font=\scriptsize] {Yes} (terminate.south);

    \draw [line] (step1) -- node[midway, above, font=\scriptsize] {No} (step2);

    \draw [line] (step2) -- (step3);

    \draw [line] (step3) -- (step4);

    \draw [line] (step4.south) |- ++(0, -0.8cm) -| (step1.south);

\end{tikzpicture}
  }
  \caption{\textbf{Flowchart of ALMTON-Heuristic.} The model phase (Step~2) escalates $\sigma$ only as needed to enforce the invariants in Proposition~\ref{prop:model-phase}.
    Step~3 applies the mixed ratio test \eqref{eq:rho}.
  Step~4 resets $\sigma$ after success, or increases it after failure, before the next iteration.}
  \label{fig:almton_flowchart}
\end{figure}

\section{Convergence Rate of \gls{almton}}\label{sec:convergence}
We organize the global convergence analysis of the ALMTON-Heuristic (\Cref{alg:ALMTON_Meta} with \Cref{alg:Strategy_Heuristic}) into four components.
Our approach follows the standard paradigm for adaptive regularization methods~\cite{cartis2011adaptive2,birgin2017worst,cartis2020concise,zhu2024global,cartis2025cubic}.

Throughout this section,~$\tilde\sigma_k$ denotes the \emph{phase-closing} value from Step~2 (the value with which $\bar x$ and $s_k$ are obtained), while $\sigma_{k+1}$ denotes the outer update in Step~4.

\begin{itemize}[leftmargin=1.2em]
  \item[] \textbf{Part 1. Monotonicity of the algorithm:} the sequence $\{f(x_k)\}$ is \emph{monotone nonincreasing}.
  See \Cref{lem:descent} (\Cref{sec:algorithm}) and \Cref{theorem:monotonicity} below (which also recalls \Cref{lemma:condition} and \Cref{lemma:gap}).
    \item[]\textbf{Part 2. Uniform upper bound on the regularization:} there exists~$\tilde\sigma_{\max}<\infty$ such that~$\tilde\sigma_k\le \tilde\sigma_{\max}$ for all $k$ (and hence $\sigma_k\le \tilde\sigma_{\max}$ as well).
    See \Cref{lemma:upper-bound-s_k} and \Cref{lemma:sigma-bound}.
  \item[] \textbf{Part 3. Uniform lower bound on the step size on success:} every successful iteration achieves a step of size at least a problem-dependent threshold, yielding a guaranteed decrease. See \Cref{lemma:lower-bound-s_k}.
  \item[] \textbf{Part 4. A bound on unsuccessful iterations:} the number of unsuccessful iterations is controlled by the number of successful ones, via a batching argument. See \Cref{lemma:iteration-bound} and \Cref{theorem:complexity}.
\end{itemize}

Combining monotonicity with the lower bound~$f_{\mathrm{low}}$ and the bounds from Parts~2-4, \Cref{theorem:complexity} establishes global convergence and an evaluation complexity bound.

We start by making explicit the model bound enforced by the model phase.
\begin{lemma}[Model value bound returned by the model phase]
  \label{lemma:condition}
  At every nonterminal ALMTON-Heuristic iteration, the model phase returns
  $(\bar x,\tilde\sigma_k)$ with $s_k:=\bar x-x_k$ such that
  \begin{equation}\label{eq:condition1}
    m_{f,x_k}(\bar x;\tilde\sigma_k)\;=\;\Phi^{3}_{f,x_k}(\bar x)\;+\;\tilde\sigma_k \|s_k\|^2
    \;\le\; f(x_k)-\frac{c}{6}\|s_k\|^2.
  \end{equation}
\end{lemma}
\begin{proof}
  This is \cref{prop:model-phase}\textup{(iii)}.
\end{proof}

We next record the immediate consequence for the cubic Taylor gap.

\begin{corollary}\label{lemma:gap}
  For every $k\ge 0$,
  \[
    \Phi^{3}_{f,x_k}(x_k)\;-\;\Phi^{3}_{f,x_k}(\bar x)
    \;\ge\;\left(\tilde\sigma_k+\frac{c}{6}\right)\|s_k\|^{2}.
  \]
\end{corollary}
\begin{proof}
  From \cref{eq:taylor model with regularization},
  \(
    m_{f,x_k}(x;\tilde\sigma_k)=\Phi^{3}_{f,x_k}(x)+\tilde\sigma_k\|x-x_k\|^2
  \).
  Evaluating at $x=x_k$ and $x=\bar x$ and subtracting, we get
  \[
    m_{f,x_k}(x_k;\tilde\sigma_k)-m_{f,x_k}(\bar x;\tilde\sigma_k)
    =\Phi^{3}_{f,x_k}(x_k)-\Phi^{3}_{f,x_k}(\bar x)-\tilde\sigma_k\|s_k\|^2.
  \]
  Use \cref{eq:condition1} to conclude the claim.
\end{proof}

\begin{theorem}[Monotonicity of ALMTON-Heuristic]\label{theorem:monotonicity}
  Let \cref{assump:1} hold and use the exact local-minimizer oracle convention
  of \cref{prop:model-phase}.
  Then, for all iterations \( k \geq 0 \), \cref{alg:ALMTON_Meta} generates a sequence \( \{x_k\} \) with:
  \begin{equation*}
    f(x_{k+1})\leq f(x_k).
  \end{equation*}
  Moreover, if the iteration is successful (\(\rho_k\ge \eta\)), then with \(s_k:=\bar x-x_k\) and the phase-closing \(\tilde\sigma_k\) we have:
  \begin{equation*}
    \begin{cases}
      f(x_{k+1}) \;\le\; f(x_k) \;-\; \eta\,l\,\|s_k\|^2, & \text{if } \tilde\sigma_k=0,\\[1mm]
      f(x_{k+1}) \;\le\; f(x_k) \;-\; \eta\,\big(f(x_k)-m_{f,x_k}(\bar x;\tilde\sigma_k)\big)
      \;\le\;f(x_k)-\eta l\|s_k\|^2, & \text{if } \tilde\sigma_k>0,
    \end{cases}
  \end{equation*}
  so the decrease is strict whenever the iteration is successful (and \(s_k\neq 0\)).
  If the iteration is unsuccessful (\(\rho_k<\eta\)), then \(x_{k+1}=x_k\) and \(f(x_{k+1})=f(x_k)\).
  Consequently, \(\{f(x_k)\}\) is monotone nonincreasing.
\end{theorem}

\begin{proof}
  We distinguish two cases according to the acceptance test in Step~3.

  \textbf{Case 1 (successful iteration: $\rho_k\ge\eta$).}
  In this case $x_{k+1}=\bar x$ and $s_k:=\bar x-x_k$.
  Recall the mixed ratio \eqref{eq:rho} with the \emph{phase-closing} regularization $\tilde\sigma_k$:
  \begin{equation*}
    \rho_k =
    \begin{cases}
      \dfrac{f(x_k)-f(\bar x)}{\,l\,\|s_k\|^2\,}, & \tilde\sigma_k=0,\\[2mm]
      \dfrac{f(x_k)-f(\bar x)}{\,f(x_k)-m_{f,x_k}(\bar x;\tilde\sigma_k)\,}, & \tilde\sigma_k>0.
    \end{cases}
  \end{equation*}
  Since $\rho_k\ge\eta$ and $\eta\in(0,1)$, we obtain in each subcase:
  \begin{equation*}
    \begin{aligned}
      \tilde\sigma_k=0:\quad
      &f(x_k)-f(x_{k+1}) \;=\; \rho_k\,l\,\|s_k\|^2 \;\ge\; \eta\,l\,\|s_k\|^2.\\[1mm]
      \tilde\sigma_k>0:\quad
      &f(x_k)-f(x_{k+1}) \;=\; \rho_k\big(f(x_k)-m_{f,x_k}(\bar x;\tilde\sigma_k)\big)
      \;\ge\; \eta\big(f(x_k)-m_{f,x_k}(\bar x;\tilde\sigma_k)\big).
    \end{aligned}
  \end{equation*}
  By~\cref{lemma:condition} (the model value bound) we have
  \(m_{f,x_k}(\bar x;\tilde\sigma_k)\le f(x_k)\), so the term
  \(f(x_k)-m_{f,x_k}(\bar x;\tilde\sigma_k)\) is nonnegative.
  Moreover,
  \begin{equation*}
    f(x_k)-m_{f,x_k}(\bar x;\tilde\sigma_k)
    = \big(\Phi^3_{f,x_k}(x_k)-\Phi^3_{f,x_k}(\bar x)\big) - \tilde\sigma_k \|s_k\|^2
  \end{equation*}
  because \(m_{f,x_k}(x;\sigma)=\Phi^3_{f,x_k}(x)+\sigma\|x-x_k\|^2\) and \(\Phi^3_{f,x_k}(x_k)=f(x_k)\).
  Applying~\cref{lemma:gap} gives
  \(\Phi^3_{f,x_k}(x_k)-\Phi^3_{f,x_k}(\bar x)\ge \tilde\sigma_k \|s_k\|^2\), hence
  \(f(x_k)-m_{f,x_k}(\bar x;\tilde\sigma_k)\ge 0\).
  By construction of Step~2, Step~3 is only formed with a bona fide step \(s_k\ne 0\) and positive denominator in the ratio, so the decrease is \emph{strict} in both subcases:
  \begin{equation*}
    \begin{cases}
      f(x_k)-f(x_{k+1}) \ge \eta\,l\,\|s_k\|^2 > 0, & \tilde\sigma_k=0,\\[1mm]
      f(x_k)-f(x_{k+1}) \ge \eta\,\big(f(x_k)-m_{f,x_k}(\bar x;\tilde\sigma_k)\big) > 0, & \tilde\sigma_k>0.
    \end{cases}
  \end{equation*}

  \textbf{Case 2: Unsuccessful iteration (\( \rho_k < \eta \)).}
  The iterate is rejected and \(x_{k+1}=x_k\).
  Therefore \(f(x_{k+1})=f(x_k)\) and \(x_{k+1}-x_k=\mathbf{0}\).

  \medskip
  Combining the two cases, we conclude \(f(x_{k+1})\le f(x_k)\) for all \(k\),
  with strict inequality on every successful iteration.
  Hence \(\{f(x_k)\}\) is monotone nonincreasing.
\end{proof}

\begin{lemma}[Uniform bound on the LM backstop]\label{lem:alpha-max}
Under \cref{assump:2},
\begin{equation}\label{eq:alpha_max}
  \alpha_{LM}(x_k)\le\alpha_{\max}
  :=\sqrt{3\Lambda_1\sqrt n\,\Lambda_3}+\Lambda_2.
\end{equation}
\end{lemma}
\begin{proof}
The vectors in \eqref{eq:alpha_LM} satisfy
$\|g_k\|=\|\nabla f(x_k)\|\le\Lambda_1$ and, because every tensor slice is
a matrix,
\[
  \|(\nabla^3f(x_k))_i\|_2\le\|\nabla^3f(x_k)\|_{[3]}\le\Lambda_3,
  \qquad \|h_k\|\le\sqrt n\,\Lambda_3.
\]
Thus $g_k^\top h_k\le\|g_k\|\|h_k\|\le
\Lambda_1\sqrt n\,\Lambda_3$.  Also
$-\min\{0,\lambda_{\min}(\nabla^2f(x_k))\}\le\Lambda_2$.  Substitution in
\eqref{eq:alpha_LM} proves the claim.
\end{proof}

\begin{lemma}[Uniform upper bound for $\|s_k\|$]\label{lemma:upper-bound-s_k}
Under \cref{assump:2}, every valid Heuristic trial satisfies
\begin{equation}\label{eq:step-upper-bound}
  \|s_k\|\le S_{\max}:=4\Lambda_1c^{-1}.
\end{equation}
\end{lemma}
\begin{proof}
Using \eqref{eq:kkt-stationarity} to eliminate the cubic term gives
\[
  m_{f,x_k}(\bar x;\tilde\sigma_k)-f(x_k)
  =\tfrac23\nabla f(x_k)^\top s_k
   +\tfrac16s_k^\top A_k(\tilde\sigma_k)s_k<0.
\]
Since $A_k(\tilde\sigma_k)\succeq cI_n$ and $s_k\ne0$,
$(c/6)\|s_k\|^2<(2/3)\|\nabla f(x_k)\|\|s_k\|$.
Canceling $\|s_k\|$ and using \cref{assump:2} proves the result.
\end{proof}

\begin{lemma}[Uniform bound on the regularization]\label{lemma:sigma-bound}
Under \cref{assump:1,assump:2}, define
\begin{equation}\label{eq:sigma-threshold}
  \chi:=\frac{16L\Lambda_1^2}{3c^2},\qquad
  \tau:=\max\{1,\alpha_{\max},c+\Lambda_2,\chi\},
\end{equation}
and
\begin{equation}\label{eq:regularization-upper-bound}
  \tilde\sigma_{\max}:=\gamma\tau.
\end{equation}
Then, at every nonterminal outer iteration $k$, its starting value and
phase-closing value satisfy
\begin{equation}\label{eq:sigma-both-bounds}
  0\le\sigma_k\le\tilde\sigma_k\le\tilde\sigma_{\max}.
\end{equation}
Moreover, every valid regularized trial with $\tilde\sigma_k\ge\chi$ is
successful.
\end{lemma}
\begin{proof}
Let $D_k:=f(x_k)-m_{f,x_k}(\bar x;\tilde\sigma_k)>0$ and
$R_k:=f(\bar x)-\Phi^3_{f,x_k}(\bar x)$.  The regularized ratio obeys
\begin{equation}\label{eq:successful-iteration-ratio-deviation}
  \rho_k-1
  =\frac{\tilde\sigma_k\|s_k\|^2-R_k}{D_k}.
\end{equation}
By \Cref{cor:p3}, $R_k\le(L/3)\|s_k\|^4$.  Hence
\Cref{lemma:upper-bound-s_k} implies $\rho_k\ge1>\eta$ whenever
$\tilde\sigma_k\ge(L/3)S_{\max}^2=\chi$.

We next control overshoot.  If an inner candidate $\tilde\sigma<\tau$, the
three entries in the Heuristic update are bounded by
\[
  \alpha_{LM}(x_k)\le\tau,\qquad
  \gamma\max\{1,\tilde\sigma\}\le\gamma\tau,\qquad
  \tilde\sigma+(c-\lambda_k(\tilde\sigma))_+
  \le\max\{\tilde\sigma,c+\Lambda_2\}\le\tau.
\]
For the last inequality, if the positive part is inactive the entry is
$\tilde\sigma<\tau$; if it is active, use
$\lambda_k(\tilde\sigma)=\lambda_{\min}(\nabla^2f(x_k))+2\tilde\sigma$
to bound the entry by $c+\Lambda_2$.
Thus the first candidate crossing $\tau$ is at most $\gamma\tau$.  A candidate
at least $\tau$ is positive, satisfies the hypotheses of
\cref{lem:positive-alpha-backstop},
and has $A_k(\tilde\sigma)\succeq cI_n$; hence the phase exits.  It is also at
least $\chi$, so the resulting outer iteration is successful.

If a failed outer iteration has $\tilde\sigma_k>0$, it must have
$\tilde\sigma_k<\chi\le\tau$.  If instead $\tilde\sigma_k=0$, its outer
update is
$\max\{\alpha_{LM}(x_k),\gamma\}\le\gamma\tau$ directly.  Thus every
failed outer update is at most $\gamma\tau$, while a successful iteration
resets the next outer value to zero.  Starting from $\sigma_0=0$, induction gives
\eqref{eq:sigma-both-bounds}, including both inner and outer overshoot.
\end{proof}

\begin{lemma}[Lower bound for a valid trial step]\label{lemma:lower-bound-s_k}
Under \cref{assump:1,assump:2}, let $B:=2\tilde\sigma_{\max}+L$.
Every valid trial step satisfies
\begin{equation}\label{eq:steplength-lower-bound}
  \|s_k\|\ge\min\left\{1,
  \frac{\|\nabla f(x_k+s_k)\|}{B}\right\}.
\end{equation}
\end{lemma}
\begin{proof}
Combining \eqref{eq:kkt-stationarity} with the gradient remainder in
\Cref{cor:p3} gives
\[
  \nabla f(x_k+s_k)=R_k^g-2\tilde\sigma_ks_k,
  \qquad \|R_k^g\|\le L\|s_k\|^3.
\]
Consequently,
$\|\nabla f(x_k+s_k)\|\le
2\tilde\sigma_{\max}\|s_k\|+L\|s_k\|^3$.
If $\|s_k\|<1$, the right-hand side is at most $B\|s_k\|$; if
$\|s_k\|\ge1$, \eqref{eq:steplength-lower-bound} is immediate.
\end{proof}
For $k\ge0$, let $\mathcal S_k:=\{0\le j\le k:\rho_j\ge\eta\}$ and define
the failure sets using the \emph{outer} value at the beginning of the iteration,
\[
  \mathcal U_k^0:=\{0\le j\le k:\rho_j<\eta,\ \sigma_j=0\},\qquad
  \mathcal U_k^+:=\{0\le j\le k:\rho_j<\eta,\ \sigma_j>0\}.
\]

\begin{lemma}[Iteration count]\label{lemma:iteration-bound}
Under \cref{assump:1,assump:2} and the Heuristic update, set
\begin{equation}\label{eq:q-definition}
  q:=\left\lceil\frac{\log\tilde\sigma_{\max}}{\log\gamma}\right\rceil.
\end{equation}
Then $q\ge1$ and ALMTON-Heuristic satisfies
\begin{equation}\label{eq:iteration-complexity}
  k+1\le(2+q)|\mathcal S_k|+1+q.
\end{equation}
\end{lemma}
\begin{figure}[tb]
  \centering
  \makebox[\linewidth][c]{%
    \hspace*{-0.08\linewidth}%
    \resizebox{0.98\linewidth}{!}{%
      \usetikzlibrary{calc,positioning}

\definecolor{iterBatchInk}{HTML}{24313A}
\definecolor{iterBatchMuted}{HTML}{4F5E68}
\definecolor{iterBatchSuccessFill}{HTML}{DDF1EA}
\definecolor{iterBatchSuccessLine}{HTML}{2F7464}
\definecolor{iterBatchZeroFill}{HTML}{FFF0CC}
\definecolor{iterBatchZeroLine}{HTML}{A66A00}
\definecolor{iterBatchPositiveFill}{HTML}{E8EDF2}
\definecolor{iterBatchPositiveLine}{HTML}{596B7A}
\definecolor{iterBatchBandFill}{HTML}{F2F7F8}
\definecolor{iterBatchRuleFill}{HTML}{F6F8FA}

\begin{tikzpicture}[
  x=1cm,
  y=1cm,
  font=\small,
  text=iterBatchInk,
  >=latex,
  iter/.style={
    rectangle,
    rounded corners=1.5pt,
    minimum height=1.00cm,
    align=center,
    inner xsep=6pt,
    line width=0.65pt
  },
  zerofail/.style={
    iter,
    draw=iterBatchZeroLine,
    fill=iterBatchZeroFill,
    minimum width=2.65cm
  },
  posfail/.style={
    iter,
    draw=iterBatchPositiveLine,
    fill=iterBatchPositiveFill,
    minimum width=1.62cm
  },
  success/.style={
    iter,
    draw=iterBatchSuccessLine,
    fill=iterBatchSuccessFill,
    minimum width=1.82cm
  },
  batch/.style={
    rectangle,
    rounded corners=1.5pt,
    draw=iterBatchSuccessLine,
    fill=white,
    minimum width=1.40cm,
    minimum height=0.68cm,
    align=center,
    line width=0.6pt
  },
  tail/.style={
    batch,
    draw=iterBatchMuted,
    fill=white,
    dashed,
    minimum width=2.15cm,
    minimum height=0.82cm
  },
  flow/.style={
    ->,
    draw=iterBatchMuted,
    line width=0.65pt,
    shorten >=2pt,
    shorten <=2pt
  },
  guide/.style={
    draw=iterBatchMuted!75,
    line width=0.45pt
  },
  panel/.style={
    font=\footnotesize\bfseries,
    text=iterBatchInk,
    anchor=west
  },
  note/.style={
    font=\scriptsize,
    text=iterBatchMuted,
    align=center
  },
  rulebox/.style={
    rectangle,
    rounded corners=2pt,
    draw=iterBatchPositiveLine!70,
    fill=iterBatchRuleFill,
    line width=0.5pt,
    minimum width=7.80cm,
    minimum height=0.72cm,
    align=center,
    font=\scriptsize,
    text=iterBatchMuted
  }
]

\node[panel] at (0,5.45)
  {\textbf{(a)} One completed batch: $|\mathcal B_i|\le q+2$};
\node[note,anchor=east] at (13.85,5.45)
  {$q=\left\lceil\log(\tilde\sigma_{\max})/\log(\gamma)\right\rceil$};

\node[rulebox] (growthrule) at (6.45,4.52)
  {$\mathcal B_i=\mathcal S_i$ \;or\;
   $\mathcal B_i=\mathcal U_i^0(\mathcal U^+)^{r_i}\mathcal S_i$,
   $0\le r_i\le q$\\[-1pt]
   after $\mathcal U^0$: $\sigma_{j+1}\ge\gamma$; \qquad
   after each $\mathcal U^+$:
   $\sigma_{j+1}\ge\gamma\tilde\sigma_j\ge\gamma\sigma_j$};

\node[zerofail] (u0) at (1.55,3.18)
  {$\mathcal U^0$\\[-1pt]
   {\scriptsize failure with starting $\sigma_j=0$}};
\node[posfail,minimum width=3.75cm] (uprun) at (6.25,3.18)
  {$(\mathcal U^+)^{r_i}$\\[-1pt]
   {\scriptsize failures with starting $\sigma_j>0$}};
\node[success] (succ) at (10.72,3.18)
  {$\mathcal S_i$\\[-1pt]
   {\scriptsize $\rho_j\ge\eta$}};

\draw[flow] (u0.east) -- (uprun.west);
\draw[flow] (uprun.east) -- (succ.west);
\draw[flow] (succ.east) -- ++(1.45,0);

\draw[guide]
  (growthrule.south) -- (6.45,3.78);

\node[note] at (1.55,2.24)
  {optional; at most one};
\node[note] at (6.25,2.24)
  {optional; $0\le r_i\le q$};
\node[note] at (10.72,2.24)
  {exactly one success};
\node[note] at (12.72,3.58)
  {reset: $\sigma_{j+1}=0$};

\draw[iterBatchMuted!35,line width=0.45pt]
  (0,1.78) -- (13.85,1.78);

\node[panel] at (0,1.55)
  {\textbf{(b)} Prefix decomposition};

\draw[rounded corners=3pt,draw=iterBatchSuccessLine!55,
      fill=iterBatchBandFill,line width=0.55pt]
  (0.22,0.12) rectangle (8.65,1.06);
\node[note,text=iterBatchSuccessLine] at (4.44,1.20)
  {$|\mathcal S_k|$ completed batches};

\node[batch] (b1) at (1.15,0.59) {$\mathcal B_1$};
\node[batch] (b2) at (3.05,0.59) {$\mathcal B_2$};
\node[font=\large,text=iterBatchMuted] (bdots) at (4.62,0.59) {$\cdots$};
\node[batch,minimum width=2.10cm] (bs) at (6.95,0.59)
  {$\mathcal B_{|\mathcal S_k|}$};
\draw[flow] (b1.east) -- (b2.west);
\draw[flow] (b2.east) -- (bdots.west);
\draw[flow] (bdots.east) -- (bs.west);

\node[note] at (4.44,-0.18)
  {each completed batch has length at most $q+2$};

\node[note] at (11.20,1.20)
  {one trailing partial batch};
\node[tail] (trail) at (11.20,0.59)
  {$\mathcal T_k$\\[-1pt]{\scriptsize possibly empty}};
\draw[flow] (8.65,0.59) -- (trail.west);
\node[note] at (11.20,-0.18)
  {$|\mathcal T_k|\le q+1$};

\node[anchor=center,font=\footnotesize] at (6.92,-0.92)
  {$\displaystyle
    k+1
    =|\mathcal S_k|+|\mathcal U_k^0|+|\mathcal U_k^+|
    \le(q+2)|\mathcal S_k|+(q+1).$};

\end{tikzpicture}%
    }%
  }
  \caption{\textbf{Outer-iteration batching for ALMTON-Heuristic.}
  Each completed batch ends in a success and may contain one zero-start failure
  followed by at most $q$ positive-start failures.  Success resets $\sigma$;
  failed positive-start iterations increase it geometrically.  One trailing
  partial batch may remain.}
  \label{fig:iteration_batches}
\end{figure}
\begin{proof}
Because $\tilde\sigma_{\max}=\gamma\tau\ge\gamma>1$, $q\ge1$.
A success resets the next outer value to zero; after a failure whose outer
value is zero, the next outer value is at least $\gamma>0$.  Therefore
\begin{equation}\label{eq:U1-bound}
  |\mathcal U_k^0|\le|\mathcal S_k|+1.
\end{equation}
During a consecutive run of failures with positive outer values, the model
phase only increases the parameter and the outer update gives
$\sigma_{j+1}\ge\gamma\tilde\sigma_j\ge\gamma\sigma_j$.  The first positive
outer value is at least $\gamma$ and every phase-closing value is at most
$\tilde\sigma_{\max}$, so such a run has length at most $q$.  Hence
\begin{equation}\label{eq:U2-bound}
  |\mathcal U_k^+|\le q(|\mathcal S_k|+1).
\end{equation}
Adding $|\mathcal S_k|$, \eqref{eq:U1-bound}, and \eqref{eq:U2-bound} proves
\eqref{eq:iteration-complexity}.  This argument uses the Heuristic update and
does not claim the same bound for an arbitrary update rule.
\end{proof}

\begin{theorem}[Evaluation complexity]\label{theorem:complexity}
Let $B:=2\tilde\sigma_{\max}+L$ and
\begin{equation}\label{eq:delta-epsilon}
  \delta_\epsilon:=\eta l\min\left\{1,\frac{\epsilon^2}{B^2}\right\}.
\end{equation}
Under \cref{assump:1,assump:2} and the exact local-minimizer oracle convention
of \cref{prop:model-phase}, ALMTON-Heuristic terminates at an iterate
$x_\epsilon$ satisfying $\|\nabla f(x_\epsilon)\|\le\epsilon$.  Its numbers
of successful and total outer iterations satisfy, respectively,
\begin{align}
  N_{\rm succ}
  &\le\left\lfloor\frac{f(x_0)-f_{\rm low}}{\delta_\epsilon}\right\rfloor+1,
  \label{eq:successful-count}\\
  N_{\rm out}
  &\le(2+q)\left(\left\lfloor
       \frac{f(x_0)-f_{\rm low}}{\delta_\epsilon}\right\rfloor+1\right)+1+q.
  \label{eq:total-count}
\end{align}
In particular, for $0<\epsilon\le B$, with
$\kappa_s:=B^2/(\eta l)$,
\[
  N_{\rm succ}\le
  \left\lfloor\kappa_s\frac{f(x_0)-f_{\rm low}}{\epsilon^2}\right\rfloor+1,
\]
and $N_{\rm out}=\mathcal O(\epsilon^{-2})$.  Conservatively, the method uses
at most $N_{\rm out}+1$ gradient evaluations, $N_{\rm out}$ joint
Hessian/third-derivative evaluations, and $N_{\rm out}+1$ objective
evaluations; all inner solves reuse the derivatives at the same $x_k$.
These are evaluation-oracle counts and do not count the inner \gls{sdp} solves.
\end{theorem}
\begin{proof}
If a successful iteration has a nonterminal successor, then
$\|\nabla f(x_{k+1})\|>\epsilon$.  Combining
\eqref{eq:unified-success-decrease} with
\Cref{lemma:lower-bound-s_k} yields
\[
  f(x_k)-f(x_{k+1})\ge\delta_\epsilon.
\]
Telescoping and $f\ge f_{\rm low}$ therefore bound the number of successful
iterations whose successor remains nonterminal by
$\lfloor(f(x_0)-f_{\rm low})/\delta_\epsilon\rfloor$.  At most one additional
successful iteration produces the terminal point, which proves
\eqref{eq:successful-count}.  If the algorithm did not terminate,
\Cref{lemma:iteration-bound} would force infinitely many successful
iterations, and the same decrease would contradict lower boundedness.
Finally, substitute \eqref{eq:successful-count} in
\eqref{eq:iteration-complexity} to obtain \eqref{eq:total-count}.  The
evaluation counts follow from Steps~1--3 and from the fact that the
base-Hessian test requires no trial-point derivative evaluation.
\end{proof}

\begin{remark}[Why quadratic LM, and cubic/quartic alternatives]
One may ask whether a higher-degree regularizer would improve the worst-case rate. A cubic term $\sigma\|s\|^3$ can preserve local fidelity better, but it renders the model \emph{nonpolynomial} and breaks the cubic-polynomial \gls{sdp}/\gls{sos} machinery underlying our subproblem solve; a quartic term $\sigma\|s\|^4$ keeps the model polynomial but turns the subproblem into a quartic minimization. A separate study on the tractability of type of cubic regularized model are given in \cite{cartis2025cubic, zhou2025tight}.
We therefore  use quadratic LM regularization to preserve the cubic structure and the unified \gls{sdp} solve. Developing tractable local-minimizer oracles for trust-region, cubic-, or quartic-regularized variants---which could in principle reach the sharper $\mathcal{O}(\epsilon^{-3/2})$ or $\mathcal{O}(\epsilon^{-4/3})$ rates---is a separate research direction.
\end{remark}

\subsection{Proof for Local Convergence}
\begin{theorem}[Local recovery of the unregularized step and cubic convergence]
\label{thm:local-cubic}
Let Assumption~\ref{assump:1} hold with \(p=3\), and let \(x^\ast\) be a
local minimizer such that
\[
H_\ast := \nabla^2 f(x^\ast) \succ 0.
\]
Set
\[
\mu_\ast := \lambda_{\min}(H_\ast)>0,
\]
and choose the curvature floor \(c\in(0,\mu_\ast/2)\). Let
\(\eta\in(0,1)\) and \(l\in(0,c/6]\) be the constants used in the
unregularized acceptance ratio.

Assume that, whenever the unregularized model is tested in the neighborhood
below, the exact model oracle returns the strict local minimizer on the local
branch associated with $x^\ast$.

Then there exist a neighborhood \(\mathcal N\) of \(x^\ast\) and a constant
\(C>0\) such that the following holds. Suppose that \(x_k\in\mathcal N\)
and that the model phase is initialized with \(\sigma_k=0\) (equivalently,
the algorithm first tests the unregularized cubic model at iteration \(k\)),
and suppose that the iteration is nonterminal, so
$\|\nabla f(x_k)\|>\epsilon$.
Then:

\begin{enumerate}
\item The unregularized cubic model
$
m_{f,x_k}(\cdot;0)=\Phi^3_{f,x_k}(\cdot)
$
has a strict local minimizer \(\bar x\). This minimizer passes the curvature
test, is accepted with \(\tilde\sigma_k=0\), and hence the update resets
\(\sigma_{k+1}=0\).

\item The accepted point satisfies
\[
x_{k+1}=\bar x,
\]
so the ALMTON iterate coincides with the unregularized third-order Newton
step. Moreover,
\[
\|x_{k+1}-x^\ast\|
\le C\|x_k-x^\ast\|^3 .
\]
\end{enumerate}

Consequently, once an accepted point $x_k$ enters \(\mathcal N\), every
subsequent nonterminal step remains in the unregularized regime and satisfies
the cubic recursion until the stopping test fires.  The associated untruncated
iteration map converges cubically to $x^\ast$.
\end{theorem}

\begin{proof}
Let
$
\mu_\ast:=\lambda_{\min}(\nabla^2 f(x^\ast))>0 .
$
Since \(\nabla^2 f\) is continuous and \(c<\mu_\ast/2\), there exists
\(r_0>0\) such that
\[
\nabla^2 f(x)\succeq \frac{\mu_\ast}{2}I \succ cI,
\qquad
\forall x\in B_{r_0}(x^\ast).
\]
In particular, \(f\) is strongly convex on \(B_{r_0}(x^\ast)\).
We give a local construction of the cubic Newton point.  On a possibly
smaller ball, let $G_2$ and $G_3$ be upper bounds for
$\|\nabla^2f\|_{[2]}$ and $\|\nabla^3f\|_{[3]}$.  For $x$ near $x^\ast$, put
$g:=\nabla f(x)$, $H:=\nabla^2f(x)$, $T:=\nabla^3f(x)$, and
$r_x:=4\|g\|/\mu_\ast$.  On $\{s:\|s\|\le r_x\}$ consider
\[
  \Psi_x(s):=-H^{-1}\!\left(g+\tfrac12T[s,s]\right).
\]
Since $\|H^{-1}\|\le2/\mu_\ast$, for $x$ sufficiently close to $x^\ast$
the map sends this ball into itself and is a contraction:
\[
  \|\Psi_x(s)\|\le\frac{r_x}{2}+\frac{G_3}{\mu_\ast}r_x^2\le r_x,
  \qquad
  \|\Psi_x(s)-\Psi_x(t)\|
  \le\frac{2G_3r_x}{\mu_\ast}\|s-t\|<\|s-t\|.
\]
It therefore has a unique fixed point $s(x)$ in this ball.  The fixed-point
equation is exactly $\nabla\Phi^3_{f,x}(x+s(x))=0$.  Shrinking the neighborhood
so that $\|x-x^\ast\|+r_x<r_0$ ensures that the entire step ball lies in
$B_{r_0}(x^\ast)$.  Shrinking it again so that
$G_3r_x<\mu_\ast/2-c$ gives
\[
  \nabla^2\Phi^3_{f,x}(x+s(x))=H+T[s(x)]\succ cI,
\]
so $N_3(x):=x+s(x)$ is a strict local minimizer.
By the cubic-minimizer trichotomy recalled after \eqref{eq:cubic-poly}, once a
cubic polynomial has a strict local minimizer it has no other local minimizer.
Thus the exact oracle returns $N_3(x)$.

Moreover, $\|g\|\le G_2\|x-x^\ast\|$ and $\|s(x)\|\le r_x$, while the
gradient Taylor remainder and local strong convexity give
\[
  \frac{\mu_\ast}{2}\|N_3(x)-x^\ast\|
  \le\|\nabla f(N_3(x))\|
  \le L\|s(x)\|^3.
\]
Consequently, for
$C_0:=(2L/\mu_\ast)(4G_2/\mu_\ast)^3$ (or any larger positive constant),
\[
  \|N_3(x)-x^\ast\|\le C_0\|x-x^\ast\|^3.
\]
We now shrink a radius \(r\le r_0\) sufficiently and set
\[
\mathcal N := B_r(x^\ast).
\]
For \(x_k\in\mathcal N\), define
\[
\bar x := N_3(x_k),\qquad s_k:=\bar x-x_k .
\]
Then \(\bar x\to x^\ast\), \(s_k\to0\), and
\(\bar x\in B_{r_0}(x^\ast)\) after shrinking $r$.  The algorithmic curvature
test is immediate from \(A_k(0)=\nabla^2f(x_k)\succ cI\).  The construction
above also gives the endpoint model Hessian
$M_k(s_k;0)=\nabla_x^2m_{f,x_k}(\bar x;0)\succ cI$.

We now prove that the unregularized step is accepted. Since \(\bar x\) is a
strict local minimizer of the cubic model, we have
$
\nabla_x m_{f,x_k}(\bar x;0)=0 .
$
Because the Hessian of a cubic polynomial varies affinely along line segments,
for \(d:=x_k-\bar x=-s_k\) and \(t\in[0,1]\),
\[
\nabla_x^2 m_{f,x_k}(\bar x+td;0)
=
(1-t)\nabla_x^2 m_{f,x_k}(\bar x;0)
+
t\nabla_x^2 m_{f,x_k}(x_k;0).
\]
After shrinking \(\mathcal N\) if necessary, both endpoint Hessians are
bounded below by \(cI\), hence the whole segment satisfies
$
\nabla_x^2 m_{f,x_k}(\bar x+td;0)\succeq cI,
t\in[0,1].
$
Using Taylor's formula along this segment and the stationarity of \(\bar x\)
for the model,
\[
m_{f,x_k}(x_k;0)-m_{f,x_k}(\bar x;0)
=
\int_0^1 (1-t)\,
d^\top
\nabla_x^2 m_{f,x_k}(\bar x+td;0)
d\,dt .
\]
Therefore,
\[
m_{f,x_k}(x_k;0)-m_{f,x_k}(\bar x;0)
\ge
\frac{c}{2}\|s_k\|^2 .
\]
In particular, the model-value condition in the heuristic model phase is also
satisfied whenever \(s_k\neq0\). If \(s_k=0\), then \(x_k\) is already the local
stationary point in this neighborhood, hence \(x_k=x^\ast\), and the algorithm
would terminate.

By the third-order Taylor remainder bound from Assumption~\ref{assump:1}, there
is a constant \(K\ge0\), for example \(K=L/3\) under the convention of
Corollary~\ref{cor:p3}, such that
\[
f(\bar x)
\le
m_{f,x_k}(\bar x;0)+K\|s_k\|^4 .
\]
Since \(m_{f,x_k}(x_k;0)=f(x_k)\), we get
\[
\begin{aligned}
f(x_k)-f(\bar x)
&\ge
m_{f,x_k}(x_k;0)-m_{f,x_k}(\bar x;0)
-
K\|s_k\|^4  \ge
\left(\frac{c}{2}-K\|s_k\|^2\right)\|s_k\|^2 .
\end{aligned}
\]
Because \(s_k\to0\) as \(x_k\to x^\ast\), and because
$
\eta l < \frac{c}{2},
$
we may shrink \(\mathcal N\) once more so that
$
K\|s_k\|^2 \le \frac{c}{2}-\eta l .
$
Hence
$
f(x_k)-f(\bar x)
\ge
\eta l\|s_k\|^2 .
$
For the unregularized ratio used by the algorithm,
\[
\rho_k
=
\frac{f(x_k)-f(\bar x)}{l\|s_k\|^2},
\]
this gives
$
\rho_k\ge \eta .
$
Thus the trial point is accepted with \(\tilde\sigma_k=0\), so
\[
x_{k+1}=\bar x,
\qquad
\sigma_{k+1}=0 .
\]
Finally, since the accepted point is exactly the unregularized third-order Newton
step,
the local cubic convergence estimate for the unregularized method gives
\[
\|x_{k+1}-x^\ast\|
=
\|N_3(x_k)-x^\ast\|
\le
C_0\|x_k-x^\ast\|^3 .
\]
Taking \(C=C_0\) proves the claimed cubic rate. Shrinking \(\mathcal N\) once more
so that \(C r^2\le 1/2\) gives
$\|x_{k+1}-x^\ast\|\le\tfrac12\|x_k-x^\ast\|$; hence the accepted iterates
remain in \(\mathcal N\) and the untruncated recursion converges to $x^\ast$.
Since every successful iteration resets \(\sigma_{k+1}=0\), the same argument
applies inductively to all subsequent nonterminal iterates.
\end{proof}
\begin{remark}[{On global versus local guarantees}]
The global result guarantees an approximate stationary point, but does not
by itself guarantee entry into the neighborhood of a prescribed minimizer.
Conditionally, if a successful iterate enters the neighborhood in
\Cref{thm:local-cubic}, every subsequent nonterminal step recovers the
unregularized third-order Newton recursion and its local cubic rate.
\end{remark}

\section{Numerical Experiments}\label{sec:numerical_experiments}
We now evaluate the numerical performance of the proposed \gls{almton} algorithm.
The implementation uses Python~3.12.7.  Experiments were conducted on a macOS
machine with an Apple M2 Pro CPU (10 cores, 3.2\,GHz) and 16\,GB of RAM.
The \gls{sdp} subproblems arising in \gls{almton} are formulated in
\verb|CVXPY| and passed to \verb|MOSEK| \cite{mosek} for numerical solution.
If numerical difficulties arise, the implementation may retry the same
mathematical \gls{sdp} after an equivalent positive rescaling of its data.  A
candidate step is mapped back to the original scale and accepted only if it
passes the prescribed residual and strict-local-minimizer checks on the
original, unscaled cubic model.

The comparison set is chosen separately for each experiment and is specified
with the corresponding protocol.  In particular, the frozen Experiment~3 set
contains only implementations covered by its current correctness and
provenance audit; it is a focused, rather than exhaustive, comparison.
Performance is reported using audited target rates, oracle counts, iteration
counts, and wall-clock diagnostics.  The AR$2$ and AR$3$ Simple comparators use
their validated non-interpolation implementations.  Experiments~1 and~3
compare the Simple and Heuristic variants of ALMTON under the documented
Stage-2 numerical policy.  Experiment~2 fixes the Simple outer strategy and
compares its Stage-1 and Stage-2 numerical policies.  The terminal gradient
tolerance is $10^{-8}$ in Experiments~1 and~3 and $10^{-6}$ in Experiment~2.
Thus ``strategy'' refers to trial grouping (Simple versus Heuristic), whereas
``policy'' refers to the finite-precision $\sigma$ schedule (Stage~1 versus
Stage~2).  The convergence theorems concern the printed Heuristic update under
the exact-oracle convention; the Simple-policy rows are empirical variants.

By default, the first \verb|MOSEK| attempt uses its native stopping tolerances; after a numerical or certificate failure, the same mathematical \gls{sdp} at the relevant scale may receive one auditable refinement attempt with tolerance $10^{-9}$.
The primary and refinement tolerances can be overridden or the refinement disabled explicitly, and the effective tolerance is recorded for every physical solve.
In either case, solver status alone is insufficient for acceptance: every returned step must pass the implementation's original-scale post-solve certificate.
The values and figures for Experiments~1 and~3 below were regenerated under
their documented common terminal tests and recorded CVXPY/MOSEK policies.
Experiment~1 is reported as an extended deterministic comparison;
its comparator contract is fixed for the results reported here.  The
Experiment~2 paper protocol is complete and contains 510 deterministic cases,
3,060 method--case runs, fresh terminal audits, and an explicit administrative
wall-clock censoring audit.  The Experiment~3 protocol-v5 record is likewise
complete and passed its source-provenance and semantic-validation checks.
The numerical evaluation is organized into three experiments:
\begin{enumerate}[leftmargin=1.5em]
  \item \textbf{Experiment 1} (\cref{subsec:experiment_1:robustness}) isolates
  initialization robustness on five two-dimensional nonconvex landscapes.  It
  compares adaptive LM globalization with UTON using common terminal audits
  and Dolan--Mor\'{e} performance profiles.
  \item \textbf{Experiment 2} (\cref{subsec:experiment_2:high-dimension})
  separates a controlled low-rank third-order scaling track from a generalized
  Rosenbrock stress track.  The tested dimensions are
  $n\in\{5,10,20,30,50,75,100\}$.
  \item \textbf{Experiment 3} (\cref{subsec:experiment_3:high-order}) separates
  a fixed-start trajectory diagnostic from a deterministic multi-start basin
  audit on the Hairpin Turn geometry.
\end{enumerate}
\paragraph{Preview of results}
Experiment~1 provides the clearest empirical evidence for the intended role of
the LM mechanism: relative to UTON, it substantially enlarges the set of tested
initial points from which the cubic-model method returns an audited stationary
point.  The same experiment also makes the computational price explicit---the
SDP-based methods are not competitive with AR$2$-Simple in wall-clock time.
In Experiment~2, both ALMTON numerical policies reach the known target on all
306 controlled LR3 cases, matching AR$3$-Interp and Newton--CG; the Stage-2
geometric backstop reduces ALMTON's recorded outer iterations and SDP requests
but does not remove the increasing cost of the SDP pipeline.  On generalized
Rosenbrock, Stage~2 improves substantially over Stage~1 but remains below the
recorded target rates of AR$3$-Interp and Strong-Newton, with high-dimensional
runs exposing administrative censoring and stiffness sensitivity.
Experiment~3 provides a complementary geometry-specific result.  On the
formal Hairpin grid, AR$2$-Simple reaches the target from all $400$ starts,
the two ALMTON variants reach it from the same $333$ starts, and AR$3$-Simple
reaches it from $240$.  Thus ALMTON has a higher audited target count than the
tested AR$3$ implementation on this geometry, but it neither matches the basin nor the
recorded cost of AR$2$-Simple.  The fixed-start paths illustrate how the
methods traverse the turn; they do not identify derivative order as the cause
of a trajectory.
Accordingly, we interpret the experiments as evidence for a stable
globalization of UTON on selected nonconvex geometries, not as a claim that
third-order information is uniformly preferable to mature second-order
regularization.

\subsection{Experiment 1: Basin Robustness on Two-Dimensional Nonconvex Landscapes}
\label{subsec:experiment_1:robustness}

\paragraph{Question and scope}
This experiment isolates the principal numerical question motivating the
adaptive LM mechanism: can it stabilize the unregularized cubic Taylor model
over substantial changes in the initial point?  An unregularized third-order
model need not possess a strict local minimizer, and UTON~\cite{silina2022unregularized} can therefore fail
before producing a usable trial step.  ALMTON instead increases the quadratic
regularization only when the current cubic model is not sufficiently well
posed or its trial is not acceptable.  We evaluate whether this mechanism
enlarges the set of initial points from which the method reaches an audited
stationary point on selected low-dimensional nonconvex geometries.  The
experiment is designed to assess \emph{basin robustness}; it is not designed to
establish state-of-the-art running time or a uniform advantage over
second-order regularization.

\paragraph{Problems and initialization protocol}
We consider Beale, Bohachevsky, Himmelblau, and McCormick, the four objectives
used in the numerical study of Silina and Zhang
\cite[Sec.~4.1]{silina2022unregularized}, and augment this paper-facing set with
the Styblinski--Tang instance available in their companion implementation
\cite{zhang2023thirdorder}.  Their definitions and initialization rectangles
are given in Appendix~\ref{app:experiment-1-details},
\Cref{tab:exp1-benchmarks}.
For each function, every method is initialized at the same endpoint-inclusive
uniform $30\times30$ Cartesian grid over the corresponding rectangle.  Thus,
each method is applied to
$5\times900=4{,}500$ deterministic function--initialization pairs, and the
seven-method comparison comprises 31,500 runs.  No initial point is selected
conditional on solver behavior, and no random replication is used.  The
rectangles define \emph{initialization sets only}; all algorithms are
unconstrained and no projection is performed.

\paragraph{Methods, parameters, and common terminal audit}
The compact comparison comprises fixed-step GD with $\alpha=0.01$ and no line
search, Newton--CG with the exact Hessian, UTON
\cite{silina2022unregularized}, ALMTON-Simple, ALMTON-Heuristic,
AR$2$-Simple \cite{cartis2011adaptive}, and AR$3$-Simple
\cite{cartis2024efficient}.  The two AR$p$ baselines use the validated
non-interpolation Simple updates.  Both ALMTON variants use
$(c,l,\eta,\gamma)=(0.1,0.01,0.1,2)$; AR$2$ uses
$(\eta_1,\eta_2,\gamma_1,\gamma_2,\sigma_0,\sigma_{\min})
=(0.01,0.95,0.5,3,1,10^{-8})$; and AR$3$ uses the corresponding
five parameters $(0.01,0.95,0.5,3,10^{-8})$ with an inner AR$2$ solve at
tolerance $10^{-9}$.  The extended compact contract fixes one GD step size and
does not include the exploratory GD($0.05$) or safeguarded-Newton variants; it
is therefore a focused rather than exhaustive comparator set.  Complete
settings are collected in \Cref{tab:exp1-settings}.

Every run is limited to 100 reported outer iterations or model phases.  To
remove method-dependent stopping conventions, we ignore the native success
flag and apply a common terminal audit: the final objective and gradient are
recomputed, and a run is successful precisely when both are finite and
$\|\nabla f(x_{\mathrm{final}})\|\le10^{-8}$.  We also recompute the Hessian;
all terminal points declared successful in this experiment satisfy the
approximate second-order necessary condition
$\lambda_{\min}(\nabla^2f(x_{\mathrm{final}}))\ge-10^{-6}$.  This check does
not by itself certify global optimality.

\paragraph{Numerical safeguards and data reuse}
Both ALMTON variants use the same Stage-2 geometric backstop for updating
$\sigma$.  When the accepted iterate is unchanged, the implementation caches
the Taylor data $(f,g,H,T)$ and reuses the calculation of the smallest Hessian
eigenvalue across regularization trials.  It evaluates model reductions in a
cancellation-resistant form and performs a fresh-gradient safeguard before
accepting a terminal trial.  The implementation also enables CVXPY's DPP
(disciplined parameterized programming) workspace.  Within a worker, it reuses
a fixed conic template whenever the SDP structure
is unchanged, thereby avoiding repeated model-graph construction and
canonicalization.  These
implementation choices limit repeated derivative formation and temporary
model allocation.

\paragraph{Performance measures}
In addition to the aggregate counts in \Cref{tab:exp1-overall}, we use
Dolan--Mor\'{e} performance profiles \cite{dolan2002benchmarking}.  Let
$\mathcal P$ be the 4,500 prescribed function--initialization pairs, let
$m_{p,s}$ be either the reported iteration count or the measured per-worker
wall time for method $s$ on $p$, and set $m_{p,s}=+\infty$ after an audited
failure.  Since at least one method succeeds on every $p\in\mathcal P$, the
ratios and profiles are
\begin{equation}
  r_{p,s}:=\frac{m_{p,s}}{\min_{s'}m_{p,s'}},
  \qquad
  \rho_s(\tau):=\frac{1}{|\mathcal P|}
  \bigl|\{p\in\mathcal P:r_{p,s}\le\tau\}\bigr|.
\end{equation}
Thus $\rho_s(1)$ is the fraction of prescribed instances on which $s$ attains
the best recorded metric, whereas the limiting plateau is its audited success
fraction.  Because the grid is deterministic and each wall time is recorded
once, $\rho_s$ is an empirical fraction over this test set, not a sampling
probability or a statement of statistical significance.  The conditional
summaries in \Cref{tab:exp1-overall} use each method's own successful subset
and may therefore exhibit survivor-set bias; the profiles retain all
prescribed instances by assigning audited failures infinite cost.

\begin{table}[t]
  \centering
  \caption{Aggregate Experiment~1 results over 4,500 starts per method.
  Iteration/phase and worker-time entries are medians $[Q_1,Q_3]$ over audited
  successes.  Counts are method-specific and times include the terminal audit,
  so both are descriptive.}
  \label{tab:exp1-overall}
  {\small
  \renewcommand{\arraystretch}{1.08}
  \setlength{\tabcolsep}{3.5pt}
  \begin{tabular}{@{}lrrr@{}}
    \toprule
    Method & \shortstack{Audited success\\$n$ (\%)}
      & \shortstack{Reported iter./phases\\median $[Q_1,Q_3]$}
      & \shortstack{Worker time (ms)\\median $[Q_1,Q_3]$} \\
    \midrule
    GD ($\alpha=0.01$) & 2,695 (59.89) & $62\ [57,69]$ & $0.36\ [0.31,0.40]$ \\
    Newton--CG         & 4,238 (94.18) & $8\ [7,9]$ & $0.39\ [0.31,0.71]$ \\
    UTON               & 2,494 (55.42) & $3\ [3,4]$ & $53.75\ [45.75,65.67]$ \\
    AR$2$-Simple       & 4,500 (100.00) & $7\ [5,9]$ & $3.57\ [2.46,5.32]$ \\
    AR$3$-Simple       & 3,435 (76.33) & $14\ [6,21]$ & $39.45\ [21.85,92.24]$ \\
    \midrule
    \textbf{ALMTON-Simple}    & 4,496 (99.91) & $6\ [4,11]$ & $94.16\ [62.55,172.10]$ \\
    \textbf{ALMTON-Heuristic} & 4,496 (99.91) & $4\ [3,4]$ & $92.75\ [62.91,170.35]$ \\
    \bottomrule
  \end{tabular}}
\end{table}

\begin{figure}[t]
  \centering
  \includegraphics[width=0.96\linewidth]{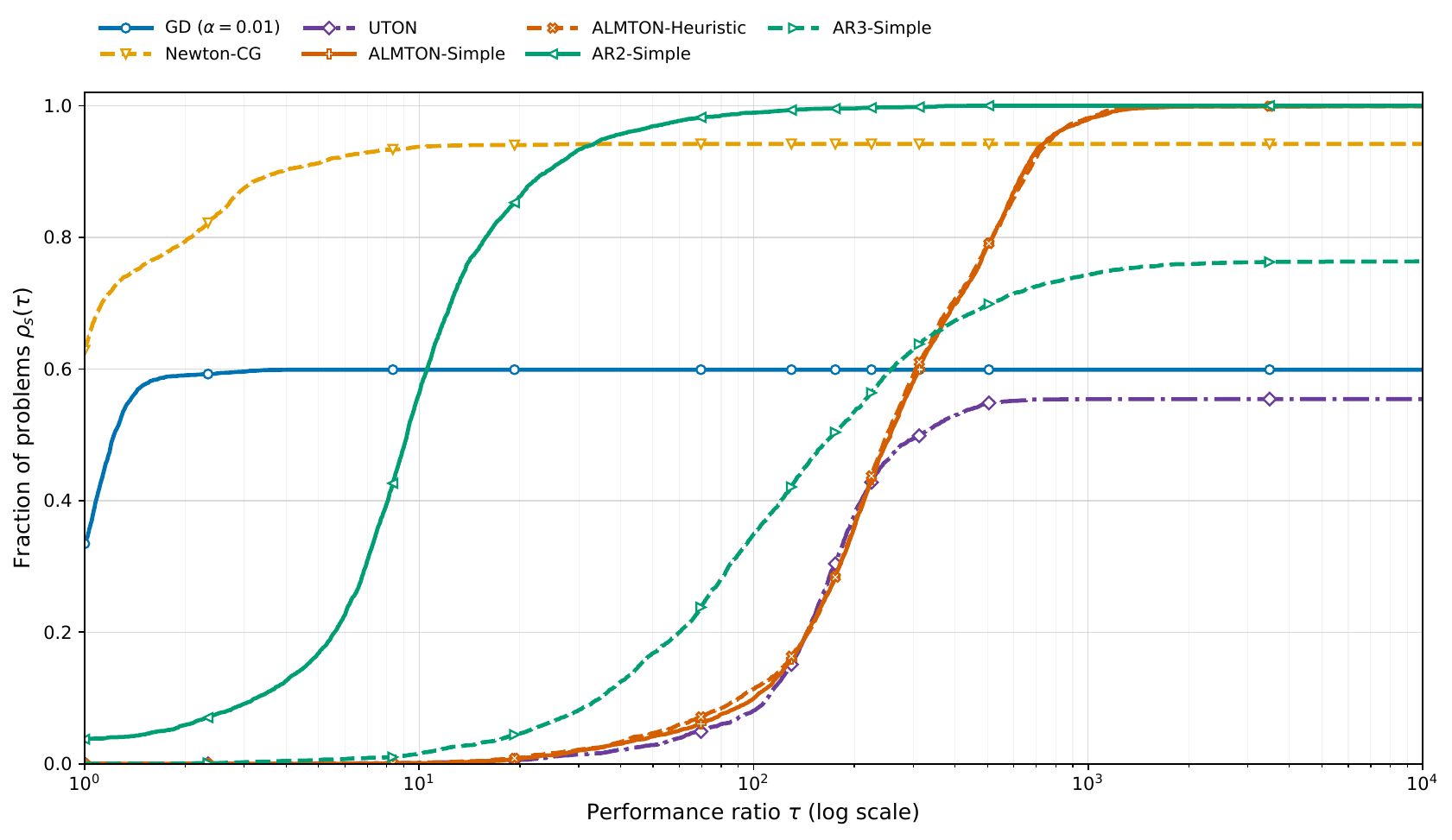}
  \caption{Wall-time Dolan--Mor\'{e} profile over the 4,500 Experiment~1
  starts.  Times include the terminal audit; each start was timed once using
  six workers and single-threaded conic solves.  The plateaus equal the audited
  success fractions.}
  \label{fig:perf_profile_time}
\end{figure}

\begin{figure}[t]
  \centering
  \includegraphics[width=0.96\linewidth]{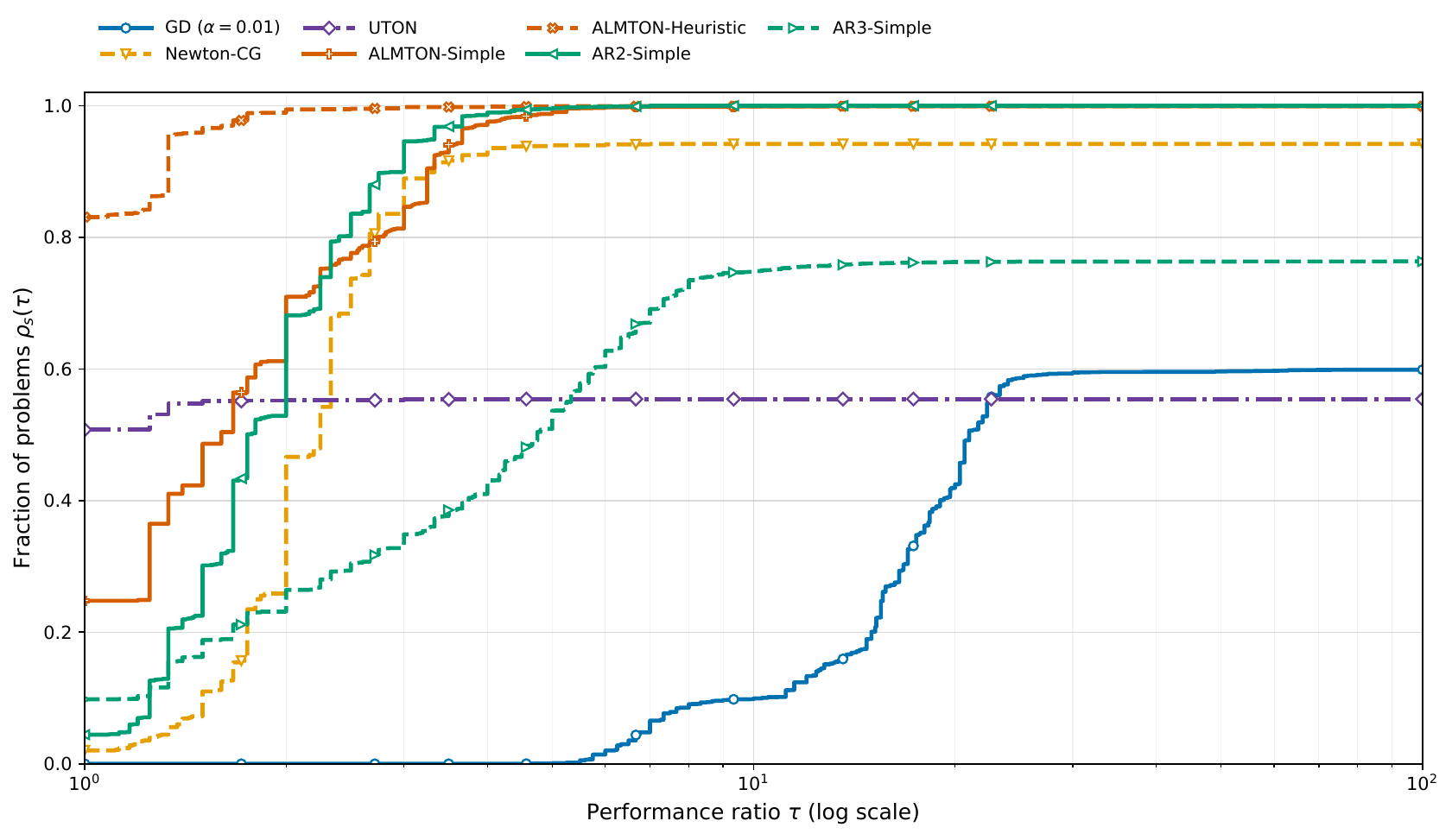}
  \caption{Iteration/phase Dolan--Mor\'{e} profile over the same 4,500 starts.
  Counts are method-specific and do not normalize inner work.}
  \label{fig:perf_profile_iter}
\end{figure}

\paragraph{Results}
Both ALMTON variants satisfy the common terminal criterion on 4,496 of the
4,500 starts (99.91\%), whereas UTON succeeds on 2,494 starts (55.42\%).  The
adaptive globalization therefore gives a net increase of 2,002 successful
starts, or 44.49 percentage points, and leaves four audited failures rather
than 2,006.
This comparison provides the central evidence from Experiment~1: the adaptive
LM mechanism stabilizes a third-order Newton construction whose unregularized
cubic model is fragile under changes of initialization.  The four ALMTON
failures all occur on Bohachevsky and arise from strict SDP residual rejection,
rather than from a change of conic solver.

The same data delimit this conclusion.  AR$2$-Simple succeeds on every start
and is substantially faster: among audited successes, its median per-worker
time, including the common terminal audit, is 3.57 ms, compared with 94.16 ms
and 92.75 ms for ALMTON-Simple and ALMTON-Heuristic, respectively.  In
\Cref{fig:perf_profile_time}, ALMTON is
never the fastest method and reaches its 99.91\% robustness plateau only at
large time ratios.  Hence the experiment does not support a claim of
state-of-the-art performance or general Pareto dominance for ALMTON; it
supports a robustness--cost tradeoff relative to
UTON.

Conditional on audited success, ALMTON-Simple and ALMTON-Heuristic report
medians of six and four outer stages, respectively, compared with fourteen for
the tested AR$3$-Simple implementation.  These are modest outer-stage counts,
but they are not equivalent measures of work.

More generally, Experiment~1 concerns five two-dimensional landscapes and does not establish
robustness for arbitrary nonconvex geometry.  Functionwise counts, benchmark
definitions, and the full numerical settings are given in
\Cref{tab:exp1-by-function,tab:exp1-benchmarks,tab:exp1-settings} of
Appendix~\ref{app:experiment-1-details}.

\subsection{Experiment 2: Controlled Scaling and Rosenbrock Stress Test}
\label{subsec:experiment_2:high-dimension}

Experiment~1 isolates robustness with respect to the initial point in two
dimensions.  Here we ask two different questions: how the present SDP-based
realization behaves as ambient dimension and active third-order rank increase,
and where its numerical policies cease to be effective on a stiff nonconvex
valley.  We therefore use two deterministic tracks.  The controlled LR3 track
separates ambient dimension, active third-order rank, coordinate
representation, and local nonconvexity; the generalized Rosenbrock track is a
stress test for stiffness, initialization, and administrative wall-clock
censoring.

\paragraph{Common protocol and terminal audit}
For every case, we run ALMTON-S1, ALMTON-S2, AR$3$-Interp
\cite{cartis2024efficient}, Strong-Newton,\footnote{Strong-Newton denotes our
safeguarded second-order Newton comparator.  It uses backtracking line search
and, when negative curvature is detected, incorporates a corresponding Hessian
eigenvector into the search direction to move away from saddle points.}
Newton--CG, and L-BFGS
\cite{nocedal2006numerical} from the same initial point, with a budget of 1,000
reported outer iterations.  S1 and S2 use the same ALMTON-Simple algorithm,
parameters, and safeguards; they differ only in the regularization update:
S1 includes $\alpha_{LM}$ immediately, whereas S2 starts from the smallest
curvature-repair value and increases $\sigma$ geometrically.  Their comparison
therefore isolates this numerical policy rather than two distinct algorithms.
The common safeguards and core ALMTON parameters are documented in
Experiment~1 and are not repeated here.

Every terminal point is re-evaluated using fresh objective, gradient, and
Hessian calls.  A run is called stationary if $f(x)$ and $\nabla f(x)$ are
finite and $\|\nabla f(x)\|\le 10^{-6}$.  Since both tracks have a known pair
$(x_\star,f_\star)$, the stricter global-target event additionally requires
\begin{equation}
  f(x)-f_\star
  \le 10^{-8}\max\{1,|f(x_0)-f_\star|\},
  \qquad
  \frac{\|x-x_\star\|}{\sqrt n}\le 10^{-4}.
  \label{eq:exp2-global-target}
\end{equation}
Headline target counts use an administrative budget of 300 seconds per task.
Censored runs count as failures; because ALMTON checks this budget at outer
iteration boundaries, a small recorded overshoot is possible.  No run first
attained the target after the limit.  The completed design contains 510
deterministic cases and 3,060 method--case runs: 306 cases per method in the LR3
track and 204 per method in the Rosenbrock track.  The record combines serial
execution with a six-worker parallel backfill.  Consequently, timings are
reported as descriptive implementation diagnostics and not as controlled
sampling distributions.

For the Rosenbrock outcome audit, a stationary nonglobal endpoint is called a
strict local minimum when
$\lambda_{\min}(\nabla^2f(x))>10^{-8}\max\{1,\|\nabla^2f(x)\|_2\}$.
Administrative censoring is recorded separately and is a subset of
nonstationary failures.

\paragraph{Track I: controlled low-rank third-order family}
Let $U\in\mathbb R^{n\times p}$ have orthonormal columns, set $y=U^\top x$, and
write $x_\perp=(I-UU^\top)x$.  We consider
\begin{equation}
  F_{n,p,\rho,U}(x)
  =\frac12\|x_\perp\|^2
   +\sum_{j=1}^{p}
    \left(\frac12y_j^2+\frac{\tau_j}{6}y_j^3+\frac14y_j^4\right),
  \qquad
  \tau_j=(-1)^{j-1}\sqrt{\rho}.
  \label{eq:exp2-lr3}
\end{equation}
The construction has $x_\star=0$ and $f_\star=0$.  We use
$n\in\{5,10,20,30,50,75,100\}$ and
$p\in\{1,2,4,8,n\}$, clipping to $p\le n$ and removing duplicates.  Identity,
block-orthogonal, and dense Haar coordinate systems control representation.
The value $\rho=4$ gives a strongly convex control initialized at
$y_0=0.6\operatorname{sign}(\tau)$; $\rho=14$ additionally admits the
nonconvex start $y_0=-0.6\operatorname{sign}(\tau)$.  Since $\rho<16$, every
case has the unique stationary point and global minimizer $x_\star=0$.  There
are 36 cases at $n=5$ and 45 at each remaining dimension.

\begin{table}[tbp]
  \centering
  \caption{Controlled LR3 summary.  Panel~A gives target attainment and
  all-run medians of iteration count $I$ and worker time $T$ over 306 cases.
  Panel~B gives selected nonconvex $\rho=14$ scaling endpoints: $R$ is the
  coincident median iteration/request/attempt count; the last column gives the
  S1 and S2 SDP-pipeline time shares.  Timings are descriptive.}
  \label{tab:exp2-lr3-targets}
  {\small
  \renewcommand{\arraystretch}{1.08}
  \setlength{\tabcolsep}{5pt}
  \begin{minipage}[t]{0.42\linewidth}
  \vspace{0pt}
  \centering
  \resizebox{\linewidth}{!}{%
  \begin{tabular}{@{}lrrr@{}}
    \toprule
    \multicolumn{4}{@{}l}{\textit{Panel A: all 306 controlled cases}}\\[-1pt]
    Method & Target & Med. $I$ & Med. $T$ (s) \\
    \midrule
    \textbf{ALMTON-S1} & $306/306$ & 13 & 2.409 \\
    \textbf{ALMTON-S2} & $306/306$ & 6 & 0.833 \\
    \midrule
    AR$3$-Interp  & $306/306$ & 3 & 0.0122 \\
    Strong-Newton & $264/306$ & 5 & 0.00115 \\
    Newton--CG    & $306/306$ & 6 & 0.00108 \\
    L-BFGS        & $250/306$ & 6 & 0.000945 \\
    \bottomrule
  \end{tabular}}
  \end{minipage}\hfill
  \begin{minipage}[t]{0.54\linewidth}
  \vspace{0pt}
  \centering
  \resizebox{\linewidth}{!}{%
  \begin{tabular}{@{}lrrrrrr@{}}
    \toprule
    \multicolumn{7}{@{}l}{\textit{Panel B: selected $\rho=14$ scaling endpoints}}\\[-1pt]
    & & \multicolumn{2}{c}{ALMTON-S1} & \multicolumn{2}{c}{ALMTON-S2}
      & \multicolumn{1}{c}{SDP time} \\
    \cmidrule(lr){3-4}\cmidrule(lr){5-6}\cmidrule(l){7-7}
    Active rank & $n$ & $R$ & $T$ (s) & $R$ & $T$ (s) & S1/S2 (\%) \\
    \midrule
    Fixed $p=4$ & 5   & 13 & 0.298   & 5 & 0.118  & $88.7/88.1$ \\
                & 100 & 11 & 42.467  & 5 & 19.530 & $96.5/96.6$ \\
    Full $p=n$  & 5   & 13 & 0.304   & 5 & 0.115  & $87.8/88.5$ \\
                & 100 & 41 & 177.235 & 5 & 23.731 & $96.9/97.1$ \\
    \bottomrule
  \end{tabular}}
  \end{minipage}}
\end{table}

\paragraph{LR3 results}
Both ALMTON policies attain the audited target in all 306 cases, including
dense rotations, full-rank third derivatives, and nonconvex starts.  This is
evidence of reliability on the constructed family, not of dominance:
AR$3$-Interp and Newton--CG also attain $306/306$, with much smaller recorded
times in this implementation.  The informative ALMTON ablation is internal.
S2 reduces the all-design median outer-iteration and logical-request counts
from 13 to 6 without changing the success set.  On the fixed-$p$ slice, the
iteration counts remain nearly constant while time grows with $n$, exposing
the ambient cost of one SDP request.  On the full-rank slice, S1 rises from 13
to 41 requests, whereas S2 remains at five or six; nevertheless, approximately
97\% of the $n=100$ wall time is still spent in the SDP-request pipeline.  Thus
the geometric backstop reduces repeated regularized models on this family but
does not make the present implementation computationally competitive or
establish an empirical complexity law.

\paragraph{Track II: generalized Rosenbrock stress test}
Let $Q$ be orthogonal and set $z=\mathbf 1+Q^\top(x-\mathbf 1)$.  The second
track uses
\begin{equation}
  R_{n,c,Q}(x)=\sum_{i=1}^{n-1}
  \left[c\bigl(z_{i+1}-z_i^2\bigr)^2+(1-z_i)^2\right],
  \qquad x_\star=\mathbf 1,\quad f_\star=0.
  \label{eq:exp2-rosenbrock}
\end{equation}
For every tested dimension, the canonical slice fixes $c=100$ and $Q=I$ and
uses the two deterministic starts
\begin{equation*}
  z_0^{\rm std}=(-1.2,1,\ldots,1),
  \qquad
  z_0^{\rm alt}=(-1.2,1,-1.2,1,\ldots),
\end{equation*}
together with five fixed gap-matched directions at each of the initial
objective gaps $10(n-1)$ and $100(n-1)$, for 12 starts in total.  At $n=20$ and
$n=50$, we augment this slice by crossing $c\in\{1,100,10^4\}$ with identity
and Haar coordinates.  Counting the already included $c=100$, $Q=I$ cells
only once, this selectively factorial design gives 204 cases per method.  The
coefficient and rotation comparisons below therefore use only balanced
subsets.

\FloatBarrier
\begin{table}[!htbp]
  \centering
  \caption{Generalized Rosenbrock outcomes.  Over all 204 cases, $G$ is
  audited global-target attainment, NGM a stationary nonglobal minimum, and
  $C\subseteq\mathrm{NS}$ administrative censoring among nonstationary
  failures.  The last three columns give $G/48$ on the balanced
  $n\in\{20,50\}$ stiffness subset.}
  \label{tab:exp2-rosen-outcomes}
  {\small
  \renewcommand{\arraystretch}{1.08}
  \setlength{\tabcolsep}{5pt}
  \begin{tabular}{@{}lrrrrrr@{}}
    \toprule
    & \multicolumn{3}{c}{All 204 cases}
    & \multicolumn{3}{c}{Balanced stiffness sweep: $G/48$} \\
    \cmidrule(lr){2-4}\cmidrule(l){5-7}
    Method & $G$ (\%) & NGM & $C$ & $c=1$ & $c=100$ & $c=10^4$ \\
    \midrule
    \textbf{ALMTON-S1} & 82 (40.20)  & 10 & 40 & $35/48$ & $18/48$ & $4/48$ \\
    \textbf{ALMTON-S2} & 123 (60.29) & 14 & 12 & $45/48$ & $34/48$ & $10/48$ \\
    \midrule
    AR$3$-Interp   & 181 (88.73) & 23 & 0 & $48/48$ & $40/48$ & $42/48$ \\
    Strong-Newton  & 171 (83.82) & 15 & 0 & $48/48$ & $36/48$ & $40/48$ \\
    Newton--CG     & 87 (42.65)  & 1  & 0 & $48/48$ & $16/48$ & $0/48$ \\
    L-BFGS         & 1 (0.49)    & 0  & 0 & $0/48$  & $0/48$  & $0/48$ \\
    \bottomrule
  \end{tabular}}
\end{table}

\FloatBarrier
\paragraph{Rosenbrock results and limitations}
S2 materially improves the ALMTON outcome over the full Rosenbrock design: the
within-budget global-target count increases from $82/204$ to $123/204$, the
stationary count increases from $92/204$ to $137/204$, and censoring decreases
from 40 runs to 12.  The improvement is not pointwise: S2 gains 45 targets
missed by S1 but loses four targets attained by S1.  It is also an internal
ALMTON improvement rather than a performance lead.  AR$3$-Interp and
Strong-Newton attain respectively $181/204$ and $171/204$ global targets; on
this test set, every S2 global target is also attained by AR$3$-Interp.
AR$3$-Interp is retained as a diagnostic legacy third-order comparator, so
these observations neither certify it as a reference AR$3$ implementation nor
support a state-of-the-art claim for ALMTON.

The canonical slice localizes the high-dimensional boundary.  S2 attains
$9/12$ targets at $n=5$, $10/12$ at each of $n=10,20,30$, and only $6/12$,
$3/12$, and $2/12$ at $n=50,75,100$.  On the balanced stiffness subset, the
S1/S2 target rates decrease from $72.9\%/93.8\%$ at $c=1$ to
$8.3\%/20.8\%$ at $c=10^4$, whereas AR$3$-Interp and Strong-Newton remain
above $80\%$ at $c=10^4$.  Moreover, 24 ALMTON runs reach stationary but
nonglobal local minima.  Thus a gradient-only success definition would
overstate recovery of the known solution, and the geometric backstop does not
remove attraction to every nonglobal basin.

All percentages are audited target-attainment rates under the stated
implementations and stopping policies, not intrinsic convergence
probabilities.  In particular, the native L-BFGS and Newton--CG stopping tests
do not coincide with the common terminal gradient gate, although their final
points are subjected to the same fresh audit.  The deterministic, selectively
factorial design, the mixed execution record, and the 52 censored ALMTON tasks
preclude a universal dimensional threshold or empirical complexity claim.
Taken together, the two tracks support a narrow conclusion: Stage~2 reduces
repeated regularized models and postpones degradation relative to Stage~1, but
the present SDP pipeline remains the dominant computational boundary and
ALMTON is not the best general-purpose method on this benchmark suite.
\FloatBarrier

\subsection{Experiment 3: Hairpin Trajectories and Basin Robustness}
\label{subsec:experiment_3:high-order}

The Hairpin Turn problem of
\cite[Example~1.1 and Appendix~A.4]{cartis2024efficient} contains a long,
narrow corridor followed by a pronounced reversal.  We use this deliberately
structured two-dimensional geometry to ask whether adaptive LM stabilization
allows the cubic-model iteration to remain operational through the turn, and
whether the resulting behavior persists under changes in the initial point.
The experiment has two complementary tracks: a fixed-start trajectory
diagnostic and a deterministic multi-start basin audit.  Neither track
identifies derivative order as the cause of an observed path or supports a
general ranking of nonconvex solvers.

We evaluate the reference Hairpin construction in the vertically reflected
coordinate $y=-y_{\rm ref}$.  If $g_{\rm hp}$ denotes the smooth turn component
defined in the reference, the implemented objective and quartic box penalty are
\begin{equation}
  \begin{aligned}
  f_{\rm hp}(x,y)
    &=g_{\rm hp}(x,y)+rx
      +50b(x;-0.4,0.5)+50b(y;-5,0),
      \qquad r=3\times10^{-4},\\[-1pt]
  b(z;a,c)
    &=\begin{cases}
      (z-a)^4, & z\leq a,\\
      0,       & a<z<c,\\
      (z-c)^4, & z\geq c.
    \end{cases}
  \end{aligned}
  \label{eq:hairpin-reflected}
\end{equation}
A high-accuracy reference computation gives the strict local minimizer
\[
  \widehat x=(-0.4608434,0.1075773)^\top,
  \qquad f_{\rm hp}(\widehat x)=-0.5190288.
\]
The displayed rectangles below specify starting points and plotting windows
only: every method is unconstrained, and no iterate is projected.

\paragraph{Common protocol}
The frozen protocol-v5 comparison contains GD with
$\alpha\in\{0.01,0.05\}$, ordinary full-step Newton, the retained
safeguarded-Newton implementation, AR$2$-Simple, AR$3$-Simple,
ALMTON-Simple, and ALMTON-Heuristic.  The validated non-interpolation
AR$2$/AR$3$ policies and the ALMTON parameters and numerical safeguards are
the same as in Experiment~1 and are not repeated here.  Each run is allowed
at most $K_{\max}=1000$ attempted outer iterations.  A target success is
recorded only if a fresh terminal audit verifies
\begin{equation}
  \|\nabla f_{\rm hp}(x_k)\|\leq10^{-8},\qquad
  \lambda_{\min}(\nabla^2f_{\rm hp}(x_k))\geq-10^{-8},\qquad
  \|x_k-\widehat x\|\leq10^{-3}.
  \label{eq:hairpin-success}
\end{equation}
The formal record passed its completion, source-hash, and semantic-validation
checks.  The comparison set is nevertheless focused rather than exhaustive,
as discussed below.  This retained safeguarded-Newton implementation is the
line-search baseline from the Experiment~1 code path and is distinct from the
Strong-Newton comparator used in Experiment~2.

\paragraph{Track I: fixed-start trajectory}
All eight methods start from the prescribed diagnostic point
$x_0=(0.5,0)^\top$, which was
fixed independently of the multi-start outcomes and is not a node of the grid
used in Track~II.  As shown in
\Cref{fig:hairpin_trajectory_comparison}, AR$2$-Simple, ALMTON-Simple, and
ALMTON-Heuristic pass \eqref{eq:hairpin-success} after 19, 60, and 23 reported
outer iterations, respectively.  AR$3$-Simple reaches the reference objective
and lies within $5.4\times10^{-9}$ of $\widehat x$, but terminates after a
regularization overflow with
$\|\nabla f_{\rm hp}(x)\|=2.59\times10^{-8}$ and therefore fails the declared
gradient gate.  Both GD choices exhaust the budget, ordinary Newton encounters
a singular initial Hessian, and the retained safeguarded-Newton implementation
stops away from the reference minimizer.  Thus the path is an illustrative
finite-budget diagnostic, not a convergence comparison by derivative order.

\FloatBarrier
\begin{figure}[!htbp]
  \centering
  \includegraphics[width=\linewidth]{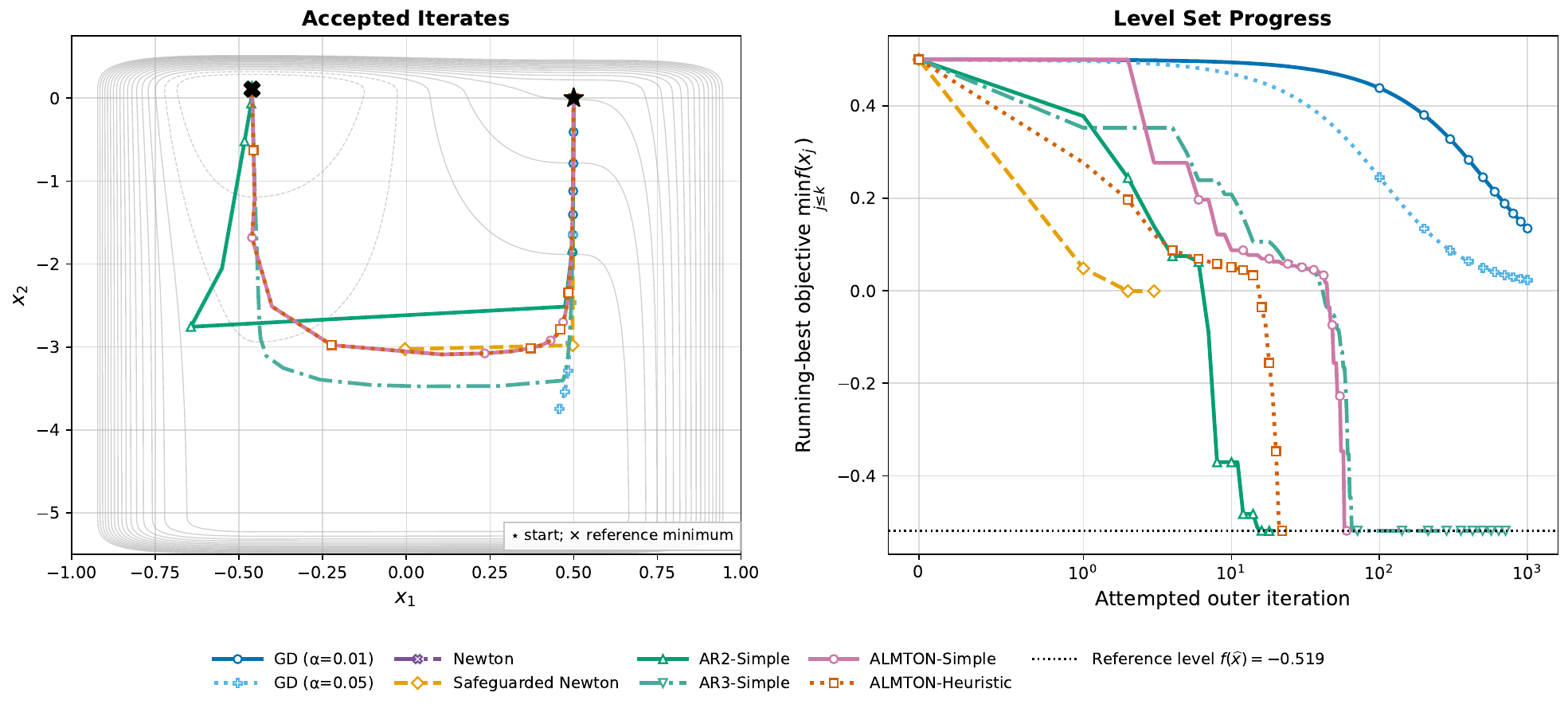}
  \caption{Hairpin Turn trajectories from $x_0=(0.5,0)^\top$.  The panels show
  accepted iterates and the running-best objective versus attempted outer
  iteration, respectively.  The dotted reference level alone does not imply
  success: AR$3$-Simple reaches it but fails the terminal gradient test in
  \eqref{eq:hairpin-success}.}
  \label{fig:hairpin_trajectory_comparison}
\end{figure}
\FloatBarrier

\paragraph{Track II: deterministic basin audit}
We use an endpoint-inclusive $20\times20$ Cartesian grid on
$[-1,1]\times[-5,0]$, giving 400 prescribed starts per method and 3,200 runs
in total, with no random replication.  The grid spans the complete reflected
corridor rather than only its entrance.  All terminal points are subjected to
the common audit \eqref{eq:hairpin-success}.  Oracle calls aggregate the
recorded objective, gradient, Hessian, and third-derivative evaluations,
including the final Hessian audit.

\begin{table}[!htbp]
  \centering
  \caption{Hairpin Turn basin audit on the common $20\times20$ grid.  Success
  is defined by \eqref{eq:hairpin-success}; ``Iters [IQR]'' is the
  success-conditioned median and IQR width.  ``Capped iters'' is the all-start
  median with failures set to $K_{\max}+1=1001$.  Time and oracle calls are
  success medians; method-specific counts and timings are descriptive.}
  \label{tab:hairpin_stats}
  \renewcommand{\arraystretch}{1.08}
  \setlength{\tabcolsep}{3pt}
  \begin{footnotesize}
    \begin{tabular*}{\textwidth}{@{\extracolsep{\fill}}lccccc@{}}
\toprule
\textbf{Solver} & \textbf{Success} & \textbf{Iters [IQR]} & \textbf{Capped iters} & \textbf{Time (s)} & \textbf{Oracle calls} \\
\midrule
\multicolumn{6}{l}{\textit{First-order baseline}} \\
GD ($\alpha=0.01$) & 84/400 (21.00\%) & 639 [336.5] & 1001 & 0.0268 & 646 \\
GD ($\alpha=0.05$) & 246/400 (61.50\%) & 256 [317.2] & 545 & 0.0138 & 263 \\
\addlinespace[2pt]
\multicolumn{6}{l}{\textit{Second-order methods}} \\
Newton & 175/400 (43.75\%) & 517 [513] & 1001 & 0.0599 & 1039 \\
Safeguarded Newton & 1/400 (0.25\%) & 7 [0] & 1001 & 0.00168 & 44 \\
AR$2$-Simple & 400/400 (100.00\%) & 15 [13] & 15 & 0.00835 & 43 \\
\addlinespace[2pt]
\multicolumn{6}{l}{\textit{Higher-order methods}} \\
AR$3$-Simple & 240/400 (60.00\%) & 29 [26] & 63.5 & 0.276 & 59 \\
ALMTON-Simple & 333/400 (83.25\%) & 29 [28] & 33.5 & 0.275 & 47 \\
ALMTON-Heuristic & 333/400 (83.25\%) & 13 [10] & 15 & 0.292 & 47 \\
\bottomrule
\end{tabular*}

  \end{footnotesize}
\end{table}
\FloatBarrier

\paragraph{Results and interpretation}
\Cref{tab:hairpin_stats} shows that AR$2$-Simple reaches the reference
minimum from all 400 starts.  ALMTON-Simple and ALMTON-Heuristic succeed from
the same 333 starts ($83.25\%$), compared with 240 ($60.00\%$) for
AR$3$-Simple.  The ALMTON and AR$3$ success sets are not nested: 123 starts
are solved only by ALMTON and 30 only by AR$3$.  AR$2$ is also markedly
cheaper in this implementation, with a median time of $0.00835$ seconds versus
$0.275$ and $0.292$ seconds for the two ALMTON variants.

The ALMTON variants have identical success masks, terminal points, oracle
records, and SDP counts.  On their 333 successes, Heuristic reports a median
of 13 outer iterations rather than 29, but it uses the same 14,609 physical
SDP solves over the full grid and is slightly slower.  Thus, the smaller outer
count reflects how regularization trials are grouped, not a measured speedup.
The 67 ALMTON failures are SDP residual-certificate rejections, whereas the
160 AR$3$ failures result from regularization overflow.

The low success rate of the retained safeguarded-Newton row is specific to its
direction and line-search implementation.  Exploratory Newton and trust-region
records were not regenerated under the frozen protocol-v5 contract and are
therefore excluded from the formal comparison.  Consequently, these results
do not support superiority claims over modern Newton or trust-region methods.
More broadly, this deterministic grid on one constructed geometry provides a
descriptive basin comparison, not a statistical or general robustness claim.
\FloatBarrier

\section{Conclusion and Future Work}
\label{sec:conclusion}

In this work, we introduced ALMTON as an adaptive LM globalization of the
unregularized third-order Newton method.  Its main advantage is a unified
global-to-local construction: quadratic regularization restores well-posedness
only when needed, while every trial step remains a cubic model handled by the
same SDP template rather than a quartically regularized subproblem.  For the
Heuristic strategy, we establish finite termination at an
$\epsilon$-first-order point and an $\mathcal O(\epsilon^{-2})$ evaluation
complexity bound under the stated assumptions and exact-oracle convention.  If
an accepted iterate enters the stated neighborhood of a positive-definite
minimizer, subsequent nonterminal steps recover the unregularized recursion and
its local cubic rate.

The experiments support the intended role of ALMTON as a robust globalization
mechanism.  In Experiment~1, it raises UTON's audited success fraction from
$55.42\%$ to $99.91\%$ over 4,500 prescribed starts.  Both S1/S2 numerical
policies attain all 306 controlled LR3 targets in Experiment~2, and on the Hairpin
geometry they reach 333 of 400 targets, compared with 240 for the tested
AR$3$-Simple implementation.  These are higher audited target counts on the
prescribed grids, not a general reliability claim.  They also do not imply
overall empirical dominance: the Rosenbrock study exposes stiffness
sensitivity, while AR$2$-Simple is faster and more successful on several tests.

The main practical limitation is the current exact-SDP realization, whose
request pipeline accounts for $96.5$--$97.1\%$ of ALMTON's recorded time on
the selected $n=100$ LR3 endpoints.  Future work will therefore investigate inexact spectral or
Krylov-based subproblem procedures, structure-exploiting tensor
representations, and broader controlled comparisons with modern Newton and
trust-region methods.  These developments should be evaluated using common
terminal audits together with arithmetic, wall-time, and peak-memory measures.

\appendix
\section{Proof of the Positive LM Backstop}
\label{app:positive-alpha-backstop}

\begin{proof}[Proof of \cref{lem:positive-alpha-backstop}]
Write $r:=\nabla f(x_k)$, $H:=\nabla^2 f(x_k)$, $T:=\nabla^3 f(x_k)$,
$K:=(\sum_i\|T_i\|_2^2)^{1/2}=\|h\|$, and
$\mu:=\lambda_{\min}(H)+2\sigma$.  If $\mu>0$ and
$\mu^2>2\|r\|K$, set $\Delta:=(\mu^2-2\|r\|K)^{1/2}$ and, when
$K>0$, $R:=(\mu-\Delta)/K$.  On the closed ball $B_R$, the map
\[
  \Psi(s):=-(H+2\sigma I)^{-1}\!\left(r+\tfrac12T[s,s]\right)
\]
maps $B_R$ into itself because
$\|T[s,s]\|^2=\sum_i(s^\top T_i s)^2\le K^2\|s\|^4$ and
$\|T[s]\|_2=\|\sum_i s_iT_i\|_2\le K\|s\|$.  Thus
$\|\Psi(s)\|\le\mu^{-1}(\|r\|+\tfrac12KR^2)=R$ for $\|s\|\le R$,
where $\|r\|+\tfrac12KR^2=\mu R$.  Brouwer's theorem therefore gives
a fixed point $s_\sigma$, and throughout $B_R$,
\[
  \nabla^2m_{f,x_k}(x_k+s;\sigma)
  =H+2\sigma I+T[s]\succeq(\mu-KR)I=\Delta I.
\]
The fixed-point equation is exactly
$\nabla m_{f,x_k}(x_k+s_\sigma;\sigma)=0$; thus $x_k+s_\sigma$ is a
strict local minimizer.  The case $K=0$ is the strictly convex quadratic
case.

It remains to verify the two inequalities used above.  Let
$a:=\sqrt{\tfrac32(\|g\|\|h\|+g^\top h)}$.  The definition
$\alpha_{LM}=a-\min\{0,\lambda_{\min}(H)\}$ implies $\mu>0$ and
$\mu^2>2\|r\|K$ whenever $\sigma\ge\alpha_{LM}$ and $\sigma>0$.
Indeed, if $\|r\|K>0$, then $\mu\ge2a$ when
$\lambda_{\min}(H)\ge0$ and
$\mu\ge2a-\lambda_{\min}(H)>2a$ otherwise, while
$4a^2\ge6\|r\|K>2\|r\|K$.  If $\|r\|K=0$, then $a=0$ and
$\alpha_{LM}=-\min\{0,\lambda_{\min}(H)\}$.  If
$\lambda_{\min}(H)\ge0$, positivity of $\sigma$ gives $\mu>0$; if
$\lambda_{\min}(H)<0$, then $\sigma\ge-\lambda_{\min}(H)$ and hence
$\mu=\lambda_{\min}(H)+2\sigma\ge-\lambda_{\min}(H)>0$.
Thus in either case $\mu^2>0=2\|r\|K$.
\end{proof}

\section{Experiment 1: Benchmark Definitions and Detailed Results}
\label{app:experiment-1-details}

This appendix gives a self-contained specification of Experiment~1.  The first
four objectives in \Cref{tab:exp1-benchmarks} reproduce the test-function
context requested by the reviewer and used by Silina and Zhang
\cite[Sec.~4.1 and Table~1]{silina2022unregularized}; Styblinski--Tang extends
their paper-facing set using an instance already present in the companion
implementation \cite{zhang2023thirdorder}.  The initialization rectangles agree with the corresponding
Newton-fractal implementation.  They are sampling rectangles, not constraints.

\begin{table}[htbp]
  \centering
  \caption{Experiment~1 benchmarks and reference minimizers.  Rectangles
  define endpoint-inclusive $30\times30$ starting grids, not constraints;
  references are used only for post hoc basin interpretation.}
  \label{tab:exp1-benchmarks}
  {\small
  \renewcommand{\arraystretch}{1.02}
  \setlength{\tabcolsep}{3pt}
  \setlength{\aboverulesep}{0.35ex}
  \setlength{\belowrulesep}{0.35ex}
  \resizebox{\linewidth}{!}{%
  \begin{tabular}{@{}
    >{\raggedright\arraybackslash}p{0.15\linewidth}
    >{\raggedright\arraybackslash}p{0.40\linewidth}
    >{\centering\arraybackslash}p{0.21\linewidth}
    >{\raggedright\arraybackslash}p{0.44\linewidth}@{}}
    \toprule
    Problem & Objective $f(x,y)$ & Initial rectangle & Reference minimizer(s) \\
    \midrule
    Beale &
    $\begin{aligned}[t]
      &(1.5-x+xy)^2\\[-2pt]
      &\quad +(2.25-x+xy^2)^2\\[-2pt]
      &\quad +(2.625-x+xy^3)^2
    \end{aligned}$ &
    $[2,3.5]\times[0,1]$ &
    $(3,0.5)$; global, $f=0$ \\
    \cmidrule(lr){1-4}
    Bohachevsky &
    $\begin{aligned}[t]
      x^2+2y^2&-0.3\cos(3\pi x)\\[-2pt]
      &-0.4\cos(4\pi y)+0.7
    \end{aligned}$ &
    $[-1,1]^2$ &
    $(0,0)$; global, $f=0$ \\
    \cmidrule(lr){1-4}
    Himmelblau &
    $(x^2+y-11)^2+(x+y^2-7)^2$ &
    $[-5,5]^2$ &
    $(3,2)$, $(-2.805118,3.131312)$,
    $(-3.779310,-3.283186)$, $(3.584428,-1.848126)$;
    global, $f=0$ \\
    \cmidrule(lr){1-4}
    McCormick &
    $\begin{aligned}[t]
      &\sin(x+y)+(x-y)^2\\[-2pt]
      &\quad -1.5x+2.5y+1
    \end{aligned}$ &
    $[-1.5,4]\times[-3,4]$ &
    $(-0.54719,-1.54719)$; box minimum,
    $f\approx-1.913223$; local $(2.5944,1.5944)$,
    $f\approx1.228370$ \\
    \cmidrule(lr){1-4}
    Styblinski--Tang &
    $\tfrac12(x^4+y^4)-8(x^2+y^2)+2.5(x+y)$ &
    $[-6,6]^2$ &
    $(a,a)$ global; $(a,b)$, $(b,a)$, $(b,b)$ local,
    where $a=-2.903534$ and $b=2.746800$ \\
    \bottomrule
  \end{tabular}}}
\end{table}

In the Beale objective, the middle residual is squared, as in the standard
definition and the companion implementation; its missing exponent in
\cite[Table~1]{silina2022unregularized} is a typographical omission.  The
McCormick objective is unbounded below; its box minimum and the other listed
points are reference points only, since all methods are unconstrained.

\begin{table}[!htbp]
  \centering
  \caption{Fixed implementation settings for Experiment~1.}
  \label{tab:exp1-settings}
  \renewcommand{\arraystretch}{0.98}
  \setlength{\tabcolsep}{4pt}
  \setlength{\aboverulesep}{0.2ex}
  \setlength{\belowrulesep}{0.2ex}
  \begin{footnotesize}
  \begin{tabular}{@{}>{\raggedright\arraybackslash}p{0.19\textwidth}>{\raggedright\arraybackslash}p{0.73\textwidth}@{}}
    \toprule
    Component & Setting \\
    \midrule
    Grid and budget & Endpoint-inclusive $30\times30$ grid on each of five functions (4,500 starts per method); 100 outer iterations/phases; no random replication. \\
    \cmidrule(lr){1-2}
    Common audit & Fresh finite $f$ and gradient, with $\|\nabla f(x_{\rm final})\|\le10^{-8}$; the Hessian audit uses $\lambda_{\min}\ge-10^{-6}$. \\
    \cmidrule(lr){1-2}
    GD; Newton--CG & GD: $\alpha=0.01$, no line search.  SciPy Newton--CG: exact Hessian, \texttt{xtol}$=10^{-8}$. \\
    \cmidrule(lr){1-2}
    UTON & Unregularized cubic Taylor model under the strict-MOSEK policy. \\
    \cmidrule(lr){1-2}
    ALMTON & Simple and Heuristic use $(c,l,\eta,\gamma)=(0.1,0.01,0.1,2)$ and the Stage-2 geometric $\sigma$ policy; the accepted-phase $\sigma$ hint is off. \\
    \cmidrule(lr){1-2}
    AR$2$-Simple & $(\eta_1,\eta_2,\gamma_1,\gamma_2,\sigma_0,\sigma_{\min})=(0.01,0.95,0.5,3,1,10^{-8})$; non-interpolation update; model tolerance $10^{-9}$. \\
    \cmidrule(lr){1-2}
    AR$3$-Simple & $(\eta_1,\eta_2,\gamma_1,\gamma_2,\sigma_{\min})=(0.01,0.95,0.5,3,10^{-8})$; non-interpolation update; inner AR$2$ tolerance $10^{-9}$. \\
    \cmidrule(lr){1-2}
    SDP policy & MOSEK~10.2.10 via CVXPY~1.4.1; one thread; no cross-solver fallback; unknown statuses rejected; original-scale certificates; refinement accuracy $10^{-9}$. \\
    \cmidrule(lr){1-2}
    Certificate tolerances & Absolute/relative: residual $10^{-9}/10^{-7}$, objective $10^{-8}/10^{-8}$, PSD $10^{-10}/10^{-8}$; strict-PD relative tolerance $10^{-10}$. \\
    \cmidrule(lr){1-2}
    ALMTON reuse & Cached Taylor data and Hessian eigenvalues; cancellation-resistant reductions, fresh-gradient safeguard, and DPP template reuse; no warm start or peak-memory claim. \\
    \cmidrule(lr){1-2}
    Execution & Python~3.12.7, NumPy~1.26.4, SciPy~1.11.4; six workers; per-worker wall time; one timing per start. \\
    \bottomrule
  \end{tabular}
  \end{footnotesize}
\end{table}

\begin{table}[!htbp]
  \centering
  \caption{Functionwise Experiment~1 results: audited successes/starts
  [median iterations over successes].  A dash means no success; iteration
  semantics are method-specific.}
  \label{tab:exp1-by-function}
  \renewcommand{\arraystretch}{1.10}
  \setlength{\tabcolsep}{4pt}
  \setlength{\aboverulesep}{0.35ex}
  \setlength{\belowrulesep}{0.35ex}
  \resizebox{\textwidth}{!}{%
  \begin{tabular}{lrrrrrr}
    \toprule
    Method & Beale & Himmelblau & McCormick & Bohachevsky & Styblinski--Tang & All \\
    \midrule
    GD ($\alpha=0.01$) & $0/900$ [--] & $895/900$ [53] & $0/900$ [--] & $900/900$ [66] & $900/900$ [62] & $2695/4500$ [62] \\
    \cmidrule(lr){1-7}
    Newton--CG & $804/900$ [11] & $900/900$ [9] & $842/900$ [7] & $852/900$ [7] & $840/900$ [7] & $4238/4500$ [8] \\
    \cmidrule(lr){1-7}
    UTON & $123/900$ [3] & $379/900$ [3] & $900/900$ [3] & $836/900$ [3] & $256/900$ [3] & $2494/4500$ [3] \\
    \cmidrule(lr){1-7}
    AR$2$-Simple & $900/900$ [9] & $900/900$ [8] & $900/900$ [5] & $900/900$ [9] & $900/900$ [7] & $4500/4500$ [7] \\
    \cmidrule(lr){1-7}
    AR$3$-Simple & $645/900$ [20] & $676/900$ [9] & $726/900$ [12] & $688/900$ [18] & $700/900$ [6] & $3435/4500$ [14] \\
    \midrule
    \textbf{ALMTON-Simple} & $900/900$ [11] & $900/900$ [5] & $900/900$ [4] & $896/900$ [5] & $900/900$ [9] & $4496/4500$ [6] \\
    \cmidrule(lr){1-7}
    \textbf{ALMTON-Heuristic} & $900/900$ [5] & $900/900$ [4] & $900/900$ [4] & $896/900$ [4] & $900/900$ [4] & $4496/4500$ [4] \\
    \bottomrule
  \end{tabular}}
\end{table}

\FloatBarrier
The four ALMTON failures are the same for both variants: the Bohachevsky starts
$(\pm1,\pm0.3103448276)$ terminate after an original-scale SDP residual
rejection.  Across all starts, Simple and Heuristic have identical terminal
points, success labels, objective, gradient, Hessian, and third-derivative counts, and
36,489 logical and 36,632 physical SDP attempts; the latter total includes 143
precision or scaling retries of the same logical subproblems.  Consequently,
the difference between their reported iteration medians in
\Cref{tab:exp1-by-function} is solely a difference in how repeated
regularization trials are grouped into phases.

\section*{Acknowledgments}
We thank Dr.~Yang Liu for helpful advice on the numerical experiments, and Dr.~Runyu Zhang for valuable comments and revision suggestions.

\scriptsize
\bibliographystyle{siamplain}
\bibliography{references}
\end{document}